\documentclass[12pt,a4paper]{amsart}
\usepackage{preamble}

\title[ergodic and foliated kernel-differentiation] 
{
ergodic and foliated kernel-differentiation method for linear responses of random systems
}

\begin{document}

\begin{abstract}

We extend the kernel-differentiation method for the linear response (parameter-derivative of averaged observables) of random dynamical systems. First, for the linear response of physical (or stationary) measures, we extend the method to an ergodic version, which is sampled by an infinitely long sample path, so it is more efficient than previous results. This is achieved by combining the likelihood ratio trick, decay of correlations, and ergodic theorem. Second, when the noise and perturbation are along a given foliation, we show that the method is still valid for both finite and infinite time. These results are derived using basic calculus via a microscopic view of transfer operators.

We use the ergodic formula to numerically compute the linear response of a tent map with additive noise. We use the foliated formula to compute the linear response of an unstable neural network with 51 layers $\times$ 9 neurons, with respect to the bias parameter. We show that adding foliated noise incurs a smaller error than adding noise in all directions, in terms of approximating the trend between the averaged observable and the parameter. We give a rough error analysis of the kernel-differentiation method and show that it can be expensive for small noise.

Finally, we derive the three basic linear response methods (path-perturbation, divergence, and kernel-differentiation methods) in simplified settings, and propose a potential future program unifying them.


\end{abstract}

\maketitle

\section{Introduction}
\label{s:intro}

\subsection{Literature review: linear response, path-perturbation method, and divergence method}
\label{s:review}

The averaged statistic of a dynamical system is of central interest in applied sciences.
The average can be taken in two ways.
When the system is random, we can average with respect to the randomness; if the system runs for a long time, we can average with respect to time.
\ree{
The long-time average can be defined for deterministic systems, and the limit measure is called the physical measure or SRB measure, and its existence is discussed in  \cite{young2002srb,srbmap,srbflow}.
}
If the system is both random and long-time, the long-time limit measure is also called the \ree{stationary measure} (we still call it the physical measure), and its existence is discussed in textbooks such as \cite{DurretText}.

We are interested in the linear response, which is the derivative of the averaged observable with respect to the parameters of the system.
The linear response is a fundamental tool for many purposes.
For example, in aerospace design, we want to know how small changes in geometry would affect the average lift of an aircraft, 
which was answered for non-chaotic fluids \cite{Jameson1988} but only partially answered for chaotic fluids \cite{Ni_CLV_cylinder}.
In climate, we want to ask how much the earth temperature would change if we perturb the $CO_2$ level \cite{GhLuc,Froyland2021}.
In machine learning, we want to extend the conventional backpropagation method to cases with gradient explosion; current practices have to avoid gradient explosion but that is not always achievable \cite{clip_gradients2,resnet}.

There are three basic methods for expressing and computing the linear response: the path perturbation method, the divergence method, and the kernel-differentiation method.
\ree{
The basic idea of the three methods can be illustrated in the one-step system in \Cref{s:3lin}.
In this review, we focus on the differences among these basic methods for many-step and infinite-step systems.
}
\lil{for}
\ree{
In particular, the pros and cons of the numerical algorithms from
the three basic methods are the same for both the many-step and infinite-step cases.
}

The path perturbation method averages the path perturbation over a lot of orbits; this is also known as the ensemble method or the stochastic gradient method (see, for example, \cite{Lea2000,eyink2004ruelle,lucarini_linear_response_climate}).
The proofs of
\ree{the path-perturbation formula}
for the physical measure of hyperbolic systems were given, for example, in \cite{Ruelle_diff_maps,Dolgopyat2004,Jiang2012}.
However, when the system is chaotic or unstable, that is, when the pathwise perturbation grows exponentially fast with respect to time, the method becomes too expensive since it requires many samples to average a huge integrand.

\lil{noise}
\ree{
The backpropagation method used in machine learning belongs to the path-perturbation method, and it encounters the same difficulty when the system is unstable, even though it runs for only a finite number of steps.
This is known as the `gradient explosion', whose avoidance is the basic concern in the designs of all neural networks.
If we can compute the linear response efficiently for general systems, we would have much more freedom in designing neural-networks.
}

The divergence method is also known as the transfer operator method, since the perturbation of the measure transfer operator is some divergence.
Traditionally, the divergence method is functional and the formula involves
\lil{distri}
\ree{
distributional derivatives (which are not pointwisely defined functions)
}
when the invariant measure is singular.
This is common for systems with a contracting direction (see, for example, \cite{Gouezel2008,Baladi2017}).
This difficulty cannot be circumvented by transformations such as the Livisic theorem in \cite{Gouezel2008}.
Hence, the traditional divergence formulas can not, or is very expensive to, be computed by a Monte-Carlo approach.
Rather, we must seek to approximate distributional derivatives on some basis functions, whose computation is cursed by dimensionality (see, for example, \cite{Galatolo2014,Wormell2019,SantosGutierrez2020,Bahsoun2018,Pollicott2000}).

\ree{
Hyperbolic systems are difficult for both methods, but it allow a nice decomposition into stable and unstable directions, and we can apply the path-perturbation method on the stable, divergence method on the unstable:
}
we see a glimpse of this strategy in the proofs mentioned above, and also in numerical attempts such as \cite{abramov2007blended}.
But all previous formulas involve exponentially growing or distributive terms, so people could not pointwisely sample the hyperbolic linear response, could not compute linear response numerically in the Monte-Carlo fashion, and could not work in high-dimensions.
The `fast response' linear response formulas solved this difficulty for deterministic hyperbolic systems.
It is made up of two parts, the adjoint shadowing lemma \cite{Ni_asl} and the equivariant divergence formula \cite{TrsfOprt}.
The fast response formula is the pointwise defined expression of the linear response without exponentially growing terms.
\nax{
It is also convenient for estimating the norm of the linear response operator on perturbations \cite{GN25}. 
The estimation given by the fast response method seems to be sharper than that given by the anisotropic Banach space method, by two orders of derivatives \cite{FP25optimalRes}.
}
The continuous-time versions of the fast response formulas are in \cite{Ni_asl,vdivF}.

\ree{
The common problem of all above methods is that they require some form of hyperbolicity.
In fact, Baladi showed that for some deterministic nonhyperbolic cases, such as the tent map, the $L^1$ linear response does not exist \cite{baladi07}.
}
\lil{leek}
\re{
But engineering applications still want to use gradients to do optimization, even though the true linear response may not exist.
So we have to modify the system so that the new system has linear response, but it is not too far from the original system; in particular, the local optimums should locate at similar places.
The linear response of the new system is then useful for optimizing the original system.
In this regard, Bahsoun and Galatolo showed that a tent map with cusps has the linear response \cite{bahsoun23}; we can interpret this as a deterministic modification of the tent map, which recovers the existence of linear responses.
An importance of their result is to show that we only need to make specific changes to specific local regions, not the entire map.
}

\subsection{Literature review: kernel-differentiation method}
\label{s:review2}

The third basic linear response method is the kernel-differentiation method for random systems.
Random dynamical systems with uniform noise are, in general, much more likely to have the linear response of physical measures than deterministic systems.
In fact, annealed linear response can exist when quenched linear response does not \cite{sedro23}.
This is because the kernel-differentiation method does not incur propagation of vectors or covectors (parameter-gradients), so it is not hindered by undesirable features of the deterministic dynamics.

There are mainly two groups of results in the kernel-differentiation method: the first group considers the linear response of physical measures.
Proofs of the existence of linear responses of physical measures for random systems were given in \cite{HaMa10,GG19,bahsoun20,ADF}.
\lil{dah}
\ree{
In particular, the kind of systems considered in the present paper (discrete time, noise having a regular distribution) are the ones considered in \cite{GG19}.
}
A main point in these proofs is to prove some uniform decay of correlation.
\ree{
The numerical results inspired by these theoretical results, such as \cite{optimal_response_23}, typically consider the linear response of physical measures, and approximate the expression on some basis function.
This can be very expensive in high dimensions.
}

\ree{
Another group of results in kernel-differentiation method typically considers finite-time systems, which is also known as the likelihood ratio method or the Monte-Carlo gradient method (see \cite{Rubinstein1989,Reiman1989,Glynn1990,PSW23}).
The main trick is to divide and multiply the kernel to recover the original distribution, so that we can sample the expression on sample paths from the original system.
This leads to Monte-Carlo-type method, which works well in high-dimensions.
For continuous-time SDEs, this is also known as the Cameron–Martin method, which is the foundation for a part of the Malliavin calculus \cite{MalliavinBook}.
The first issue for this group of results is that, when applied to physical measures of infinite-time systems,  the scale of the estimator is proportional to the orbit length, which can be too large.
}


Another issue in the previous likelihood ratio method is that we have to restart a new orbit very often. For each orbit, the observable function is only evaluated at the last step, while the earlier steps are somewhat wasted waiting for the orbits to land onto the physical measure.
This paper solves this problem by restricting our data to one single long orbit.
\lil{ga}
\ree{
  This issue is explained with more words at the end of \Cref{s:ergodic}.
}

Moreover, sometimes we are interested in the case where some directions in the phase space are of special importance, such as the level sets in the Hamiltonian system.
For these cases, the perturbation in the dynamics and the added noise are sometimes aligned to a subspace.
This paper extends the kernel-differentiation method to foliated spaces.

This paper is partly motivated by optimizing deterministic systems using linear response computed on a modified system.
We modify the deterministic system by adding noise, and then we use the kernel-differentiation method to compute the linear response of the new system. 
We shall see that such an approximation is meaningful for optimizing a model problem in machine learning.
In particular, for the purpose of optimizing a deterministic system, we shall see that if the system is foliated, then adding foliated noise yields a better approximation than adding uniform noise.

The next two topics following this paper, are to combine the kernel-differentiation method with the path-perturbation method (see \cite{dud}), and with the divergence method (see \cite{divKer}).
For discrete-time systems, this will help further reduce the size of the estimator.
For SDEs, when the diffusion coefficient depends on the parameter or the location, this new technique is necessary, and the pure kernel-differentiation method cannot work.

\subsection{Main results}
\label{s:results}

We consider the random dynamical system 
\lil{sho}
\ree{
in phase space (or state space) $\R^M$.
}
\[X_{n+1}=f^\gamma(X_n)+Y_{n+1}.\] 
Let  $h_n$ and $p_n$ be the probability density for $X_n\in \R^M$ and $Y_n\in \R^M$.
\lil{bie}
\reeout{
We need to know the expressions of $p_n$, but we do not need expressions of $h_n$.
}
We denote the density of the physical measure by $h$, which is the limit density of the dynamics.
Here $\gamma$ is the parameter that controls the dynamics and, hence, $h_n$ and $h$; by default $\gamma=0$.
We denote the perturbation $\delta(\cdot):=\partial (\cdot)/ \partial\gamma|_{\gamma=0}$, and when $f$ is invertible, define
\[\tf:=  f^\gamma \circ f^{-1}.\]
\lil{nian}
\ree{
$\Phi$ is a fixed $C^2$ observable function.
}

Through basic calculus, \Cref{s:1step} \textit{re}derives the kernel-differentiation formula for the linear response in one step.
\Cref{s:intui} gives a pictorial explanation of the main lemma.
In this simplest setting, we can see more easily how the kernel-differentiation method relates to other methods, and why it can be generalized to foliated situations.
\Cref{s:manystep} \textit{re}derives the formula for finitely many steps via
the measure-transfer-operator point of view, which is more convenient for our extensions.

The first contribution of this paper is \Cref{l:kd4h} in \Cref{s:infinite}, a kernel-differentiation formula for physical measures; it allows Monte Carlo computation in high dimensions, and the estimator is smaller than previous likelihood ratio method. 
It combines the decorrelation result proved in \cite{GG19} (and also \cite{HaMa10,bahsoun20,ADF}) and the likelihood ratio trick invented in \cite{CM44,MalliavinBook,Rubinstein1989,Reiman1989,Glynn1990}.
For this result, we assume some regularities of basic quantities, uniform convergence to physical measure, and uniform decay of correlation.

The second contribution of this paper, given in \Cref{s:ergodic}, is to improve the efficiency of generating samples by deriving the ergodic kernel-differentiation formula for physical measures.
With ergodic assumption, we can invoke the ergodic theorem to reduce the required data to one long orbit.
Hence, we only need to spend one spin-up time to land onto the physical measure, and then all data are fully utilized.
This improves the efficiency of the algorithm.
More specifically, 
let $dp$ denote the differential (this is basically the gradient) of $p$,
we prove

\begin{restatable} [orbitwise kernel-differentiation formula for $\delta h$]{theorem}{goldbach}
\label{t:dh}
Under \Cref{ass:basic,ass:hdh,ass:ergo}, we have
\[ \begin{split}
  \delta \int \Phi(x) h (x) dx
  \,\overset{\mathrm{a.s.}}{=}\, 
  \lim_{W\rightarrow\infty} \lim_{L\rightarrow\infty} - \frac 1L \sum_{n=1}^W  \sum_{l=1}^L \left( \Phi(X_{n+l}) - \Phi_{avg} \right) \, \delta f^\gamma X_l \cdot \frac{dp}p(Y_{1+l})
\end{split} \]
almost surely according to $X_0, Y_1, Y_2,\cdots \sim h\times p\times p \times \cdots$.
Here $\Phi_{avg}:=\int\Phi h$.
\end{restatable}

Here $W$ and $L$ are two separate and sequential limits, so we can generally choose $W\ll L$.
Our formula runs on only one orbit; hence, it is faster than previous likelihood ratio methods.
\Cref{s:generalize} generalizes the result to cases where randomness depends on location and parameters.

\ree{
The third contribution of this paper is: \Cref{s:foliation} extends the kernel-differentiation method to the case where the phase space is partitioned by submanifolds, 
}
and noise and perturbation are along the submanifolds (hence the noise is singular with respect to the Lebesgue measure).
\lil{sushi}
\ree{
Here we consider a more general phase space, $\cM$, an $M$-dimensional manifold.
}
Let $F = \{F_\alpha\}_{\alpha\in A}$ be a dimension-$c$ foliation on $\cM$, which is basically a collection of $c$-dimensional submanifolds partitioning $\cM$, and each submanifold is called a `leaf'.
Note that we do \textit{not} require $f$ to respect the foliation, that is, $x$ and $fx$ can be on two different leaves: This is more general than systems restricted on one submanifold.
Let $\mu_n$ be the measure for $X_n$.
Let $p(z, x_{n+1})$ be the density of probability of $X_{n+1}$, conditioned on $f(X_n)=z$, with respect to the Lebesgue measure on $F_\alpha(z)$.
The structure of \Cref{s:foliation} is similar to \Cref{s:derive}, and the main theorem is

\begin{restatable}[centralized foliated kernel-differentiation formula for $\delta \mu_T$ of time-inhomogeneous systems, comparable to \Cref{l:dhTn}] {theorem}{silverbach}
\label{t:dhTnfoli}
Under \Cref{ass:basic,ass:folibasic} for all $n$, we have
\[ \begin{split}
\delta \int \Phi(x_T)  d \mu_T(x_T)
=\E \left[ \left(\Phi_T(x_T) - \Phi_{avg,T} \right)
\sum_{m=1} ^ {T} \left( \delta \tf_m(z_m) \cdot 
\frac{d_z p_m(z_m, x_m)}{ p_m(z_m, x_m)} \right)
\right]
\\
:= \int_{x_0\in \cM} \int_{x_1\in F^1_\alpha(z_1)} \cdots \int_{x_T\in F^T_\alpha(z_T)}
\left(\Phi_T(x_T) - \Phi_{avg,T} \right)
\\
\sum_{m=1} ^ {T} \left( \delta \tf_m(z_m) \cdot 
\frac{d_z p_m(z_m, x_m)}{ p_m(z_m, x_m)} \right)
p_T(z_T, x_T)dx_T \cdots p_1(z_1, x_1) dx_1 d\mu_0(x_0).  
\end{split} \]  
Here $\Phi_{avg,T} := \int \Phi_T d \mu_T $.
Note that here $F, \tf, f, p$ can be different for different steps.
\end{restatable}

\Cref{s:algorithm} considers numerical realizations.
\Cref{s:procedure} gives a detailed list for the algorithm.
Note that the foliated case uses the same algorithm; the difference is in the setup of the problem.
\Cref{s:tent} illustrates the ergodic version of the algorithm on the tent map with additive noise.
\Cref{s:NN} illustrates the foliated version of the algorithm on a unstable neural network with 51 layers $\times$ 9 neurons.

\lil{bus}
\ree{
The fourth contribution of this paper is: investigating the potential of adding noise as a modification to deterministic systems for the purpose of optimizations.
}
In examples of \Cref{s:tent,s:NN}, the deterministic part does not have a useful linear response.
But after adding noise and using the kernel-differentiation algorithm, we get the derivative of the modified system, which is a reasonable reflection of the parameter-observable relation for the original system: this is useful for optimization of deterministic systems.
Moreover, \Cref{s:NN} shows that 
\nax{
the perturbation induced by the bias parameter in the neural network is foliated, so we can add noise along the foliation, and use our foliated formula to compute the linear response.
}
Then, we show that adding a low-dimensional noise along the foliation induces a much smaller error than adding noise in all directions.

\ree{
The fifth contribution of this paper is that \Cref{s:costErr} gives a rough estimate of the cost-error of the kernel-differentiation algorithm, which reminds us that the small-noise limit is very expensive.
}
We consider both the case where randomness is intrinsic in our model, and the case where we add randomness to deterministic models as an approximation.
This also helps to set some constants in the algorithm.

The sixth contribution of this paper is to distinguish and propose to unify three basic linear response methods. 
More specifically, \Cref{s:3lin} derives the three basic linear response formulas in one-step, to help the readers see the difference and the connections.
Then, \Cref{s:unify} proposes a program that combines the three basic formulas
(path-perturbation, divergence, and kernel-differentiation) 
to obtain a good approximation of the linear responses of deterministic systems.

The paper is organized as follows.
\Cref{s:derive} derives the kernel-differentiation formula where the noise is in all directions.
\Cref{s:foliation} derives the formula for foliated perturbations and noises.
\Cref{s:algorithm} gives the procedure list of the algorithm based on the formulas and illustrates two numerical examples.
\Cref{s:discuss} discusses some related issues.

\section{Deriving the kernel-differentiation linear response formula}
\label{s:derive}

\subsection{Preparations: notations on probability densities}
\label{s:prep}

Let $f^\gamma$ be a family of $C^2$ maps on the Euclidean space $\R^M$ of dimension $M$ parameterized by $\gamma$.
Assume that $\gamma \mapsto f^\gamma$ is $C^1$ from $\R$ to the space of $C^2$ maps.
The random dynamical system in this paper is given by 
\begin{equation} \begin{split} \label{e:dynamic}
  X_n =  f^\gamma (X_{n-1}) + Y_n ,
  \quad \textnormal{where} \quad 
  Y_n \iid  p.
\end{split} \end{equation}
The default value is $\gamma=0$, so $f:=f^{\gamma=0}$.
In this section, we consider the case where $p$ is any fixed $C^2$ probability density function.
The next section considers the case where $p$ is smooth along certain directions but singular along other directions.

\lil{run}
\ree{
We shall prove our result starting from finite-time cases, where we only consider the change of the probability density at a particular time step.
These intermediate results also apply to time-inhomogeneous cases, where $f$ and $p$ are different for each step.
}
For the time-inhomogeneous case, the dynamic is given by
\begin{equation} \begin{split} \label{e:dynamicinhomo}
  X_n =  f^\gamma_{n-1}(X_{n-1}) + Y_n ,
  \quad \textnormal{where} \quad 
  Y_n \sim  p_n.
\end{split} \end{equation}
Here, $X_n\in \R^{M_n}$ are not necessarily of the same dimension.
We shall exhibit the time dependence in \Cref{l:dhTn} and also in the numerics section.
On the other hand, for the infinite-time case in \Cref{t:dh}, we want to sample by only one orbit; this requires that $f$ and $p$ be repetitive among steps.

We define $h_T$ as the density of pushing-forward the initial measure $h_0$ for $T$ steps.
\[ \begin{split}
  h_T:=  ( L_p L_{f^\gamma})^T h_0.
\end{split} \]
Here $h_T$ depends on $\gamma$ and also on $h_0$; $L_f$ is the measure transfer operator of $f$,
which are defined by the integral equality
\begin{equation} \begin{split} \label{e:Lh}
  \int \Phi(x) \, (L_f h)(x) dx
  := 
  \int \Phi(fx)  h(x) dx.
\end{split} \end{equation}
Here $\Phi\in C^2(\R^M)$ is any fixed observable function.
\re{
The equivalent pointwise definition is,
\[ \begin{split}
(L_f h) (x) := \sum_{y\in f^{-1}x} \frac h { |Df| }(y) dy.
\end{split} \]
Note that $f^{-1}x$ is a set of points when $f$ is not injective.
}

For density $q$, $L_p q$ is pointwisely defined by convolution with density $p$:
\begin{equation} \begin{split} \label{e:Lppt}
  L_p q (x) 
  = \int q(x-y) p(y) dy
  = \int q(z) p(x-z) dz.
\end{split} \end{equation}
We shall also use the integral equality
\begin{equation} \begin{split} \label{e:Lpint}
  \int \Phi (x) L_p q (x) dx
  = \int p(y) \left( \int \Phi(x) q(x-y)  dx \right) dy
  \\
  \,\overset{z=x-y}{=}\,
  \iint \Phi(y+z) q(z) p(y) dz dy.
\end{split} \end{equation}
If we want to compute the above integration by Monte-Carlo method, then we should generate i.i.d. $Y_l=y_l$ and $Z_l=z_l$ according to density $p$ and $q$, then compute the mean of $\Phi(y_l+z_l)$.

To consider the differentiability of $h$ with respect to $\gamma$, we make the following simplifying assumptions about the differentiability of the basic quantities.
\re{
We are not optimal in terms of the regularity assumptions.
For example, the assumption $\Phi\in C^2$ in \Cref{ass:basic} is to match the assumption on $\phi$ in \Cref{ass:hdh}.
But in previous works such as \cite{GG19}, \Cref{ass:hdh} can be proved from many cases, even with less regularity in $\phi$, so we can reduce the regularity assumption for $\Phi$ in \Cref{ass:basic}.
}
\begin{assumption}
\label{ass:basic}
The densities $p,q,h$ have bounded $C^2$ norms;
$\gamma \mapsto f^\gamma$ is $C^1$ from $\R$ to the space of $C^2$ maps;
the observable function $\Phi$ is fixed and is $C^2$.
\end{assumption}

The linear response formula for finite time $T$ is an expression of $\delta h_{\gamma,T}$ by $\delta f^\gamma$, and 
\[ \begin{split}
  \delta(\cdot)
  :=\left. \frac {\partial (\cdot)}{\partial \gamma} \right|_{\gamma=0}.
\end{split} \]
Here, $\delta$ may as well be regarded as small perturbations.
In our notation,
\[ \begin{split}
  \delta f(x):=
  \delta f^\gamma(x):=
  \left.\pp{}{\gamma} f^\gamma(x)\right|_{\gamma=0} \in T_{fx} \R^M;
\end{split} \]
that is, $\delta f(x)$ is a vector at $fx$.
The linear response formula for finite-time is given by the Leibniz rule,
\begin{equation} \begin{split} \label{e:xu}
  \delta h_T
  = \delta (L_p L_{f^\gamma} )^T h_0
  = \sum_{n=0}^{T-1} ( L_p L_f )^{n} \, \delta(L_p L_{f^\gamma} ) \, (L_p L_f )^{T-1-n} h_0.
\end{split} \end{equation}
We shall give an expression for the perturbative transfer operator $\delta (L_p  L_{f^\gamma})$ later.

For the case where $T\rightarrow\infty$, 
\re{
we let $h^\gamma$ be a stationary density such that
}
\[ \begin{split}
L_p L_{f^\gamma} h^\gamma = h^\gamma
\end{split} \]
The corresponding measure $\mu^\gamma$ is called the 
\ree{stationary measure,} or physical measure.
We also define the correlation function,  
\begin{equation} \begin{split} \label{e:corre}
\E^\gamma[\phi(X_n)\psi(X_0,X_1)]
:= \int \phi(x_n) \psi(x_0,x_1) \mu^0(dx_0) p(y_1) dy_1 \cdots p(y_n) dy_n,
\end{split} \end{equation}
where $
x_n =  f^\gamma_{n-1}(x_{n-1}) + y_n
$ is a function of the dummy variables $x_0, y_1, y_2, \ldots$

For the integrated result in \Cref{l:kd4h}, we make the following assumptions for infinitely many steps.
\lil{life}
\ree{
Correspondingly, $\phi$ below will be substituted by $\Phi$ in \Cref{l:kd4h}, and $\psi(x_0,x_1)$ by $\delta f^\gamma x_{0} \cdot \frac {dp}{p} (y_{1})$,
where $y_1 = x_1-f x_0$.
}

\begin{assumption} \label{ass:hdh}
For a small interval of $\gamma$ containing zero,
\begin{enumerate}
    \item The physical density $h^\gamma$ is the weak$*$ limit of evolving $h^0$ under the dynamics with parameter $\gamma$,
      \ree{
\[ \begin{split}
  \lim_{T\in\N,\, T\rightarrow \infty} (L_p L_{f^\gamma})^T h^{0} = h^\gamma.
\end{split} \]
      }
    \item For any observable functions $\phi,\psi\in C^2$, the following sum of correlations 
\[
\sum_{n\ge1 }\left|\E^\gamma[\phi(X_n)\psi(X_0,X_1)]
- \E^\gamma[\phi(X_0)] \E^\gamma[\psi(X_0,X_1)]\right|
\]
\end{enumerate}
converges uniformly in terms of $\gamma$.
\end{assumption}

\Cref{ass:hdh} can be proved from more basic assumptions, which are typically more forgiving than the hyperbolicity assumption for deterministic cases.
For example, under \Cref{ass:basic}, the proof of the unique existence of and convergence to the stationary measure, where the system has additive noise with smooth density, can be found in probability textbooks such as \cite{DurretText}.
\cite{GG19} proves the decay of correlations when, roughly speaking, $f^\gamma$ has bounded variance and is mixing and that the transfer operator is bounded.
The technical assumptions and the rigorous proof for linear responses can also be found in \cite{HaMa10}.

\ree{
To sample the formula on one path in \Cref{t:dh,t:foli}, we need to make
}

\begin{assumption} \label{ass:ergo}
The physical measure $\mu^0$ of $\gamma=0$ is ergodic.
\end{assumption}

\Cref{ass:ergo} follows from some more basic assumptions, for example, when $p$ is supported on the entire phase space, we immediately have ergodicity (see \cite[example 6.1.6]{DurretText} for a proof in discrete space).
Also, in some engineering applications, there may be several ergodic components, but we could only see one of them, since some other practical concerns have (more or less) determined where we start our paths from; so we may have a delusion of ergodicity, due to ignorance of other initial conditions.
Finally, even when there are multiple ergodic components, our one-orbit formula improves the efficiency for sampling each component.

\subsection{One-step kernel-differentiation}
\label{s:1step}

We give a pointwise formula for $\delta (L_p L_{f^\gamma} q) (x)$, which seems to be a long-known wisdom.
An intuitive explanation of the following lemma is given in \Cref{s:intui}.
Then we give the integrated version which can be computed by Monte Carlo algorithms.
An intuition of the lemmas is given in \Cref{s:intui}.

\begin{lemma} [local one-step kernel-differentiation formula]
\label{l:local1step}
Under \Cref{ass:basic}, then for any density $h_0$
\[ \begin{split}
  \delta (L_p L_{f^\gamma} h_0) (x_1)
  = \int -  \delta f^\gamma (x_0)\cdot dp(x_1 - f x_0) \, h_0(x_0) dx_0.
\end{split} \]
Here $\delta:=\partial/\partial \gamma|_{\gamma=0}$, 
and $\delta f^\gamma (x_0) \cdot dp(x_1 - f x_0)$  is the derivative of the function $p(\cdot)$ at $x_1-fx_0$ in the direction $\delta f^\gamma (x_0)$, which is a vector at $fx_0$.
\end{lemma}

\begin{remark*}[]
Note that we do not compute $\delta L_{f^\gamma} $ separately.
In fact, the main point here is that the convolution with $p$ allows us to differentiate the kernel $p$, which is typically much more forgiving than differentiating the dynamics $f^\gamma$.
\end{remark*}

\begin{proof}
First write a pointwise expression for $L_p L_{f^\gamma} h_0$: by the definition of $L_p$ in \Cref{e:Lppt},
\[ \begin{split}
  (L_p L_{f^\gamma} h_0) (x_1)
  = \int (L_{f^\gamma} h_0) (z_1) p(x_1-z_1) dz_1.
\end{split} \]
Since $p\in C^2(\R^M)$, we can substitute $p$ into $\Phi$ in the definition of $L_f$ in \Cref{e:Lh}, to get
\[ \begin{split}
  = \int  h_0(x_0) p(x_1- f^\gamma x_0) dx_0.
\end{split} \]
Differentiating with respect to $\gamma$, we have
\[ \begin{split}
  \delta (L_p L_{f^\gamma} h_0) (x_1)
  = \int - h_0(x_0) \delta f^\gamma (x_0) \cdot dp(x_1 - f x_0) dx_0.
\end{split} \]
\end{proof}

\begin{lemma}[integrated one-step kernel-differentiation formula with likelihood ratio] \label{l:niu}
Under \Cref{ass:basic}, we have
\[ \begin{split}
  \int \Phi(x_1) \delta (L_p L_{f^\gamma} h_0) (x_1) dx_1
  = \iint - \Phi(x_1) \delta f^\gamma (x_0) \cdot \frac {dp(y_1)}{p(y_1)} \, h_0(x_0) dx_0 \, p(y_1) dy_1.
\end{split} \]
Here, $x_1$ in the right expression is a function of the dummy variables $x_0$ and $y_1$, that is, $x_1=y_1+fx_0$.
The expression $\frac{dp}{p}$ is called the `likelihood ratio'.
\end{lemma}

\begin{remark*}
The point of the likelihood ratio trick is that if we want to integrate the right expression by Monte-Carlo, just generate random pairs of $ \left\{x_{l,0}, y_{l,1}\right\} $, compute $x_{l,1}$ accordingly, and average $I_l := \Phi(x_1)\delta f^\gamma x_{l,0} \cdot \frac {dp}{p}( y_{l,1})$ over many $l$'s.
\end{remark*}

\begin{proof}
  Substitute \Cref{l:local1step} into the integration, we get
\[ \begin{split}
  \int \Phi(x_1) \delta (L_p L_{f^\gamma} h_0) (x_1) dx_1
  = - \iint \Phi(x_1) \delta f^\gamma (x_0) \cdot dp(x_1 - f x_0) \, h_0(x_0) dx_0 dx_1
\end{split} \]
The problem with this expression is that, should we want Monte-Carlo, it is not obvious which measure we should generate $x_1$'s according to.
To solve this problem, change the order of the double integration to get
\[ \begin{split}
  = - \int \delta f^\gamma(x_0)\cdot \left( \int \Phi(x_1) dp(x_1 - f x_0)  dx_1 \right) h_0(x_0) dx_0.
\end{split} \]
Change the dummy variable of the inner integration from $x_1$ to $y_1=x_1-fx_0$, 
\[ \begin{split}
  = - \int \delta f^\gamma(x_0)\cdot \left( \int \Phi(y_1+fx_0) dp(y_1) dy_1 \right) h_0(x_0) dx_0 \\
  = - \iint \Phi(x_1(x_0,y_1))\, \delta f^\gamma(x_0)\cdot \frac {dp}p (y_1) \, p(y_1) dy_1 h_0(x_0) dx_0.
\end{split} \]
Here $x_1$ is a function as stated in the lemma.
\end{proof}

\subsection{An intuitive explanation}
\label{s:intui}

We give an intuitive explanation for \Cref{l:local1step} and \Cref{l:niu}.
This will help us to generalize the method to the foliated situation in \Cref{s:foliation}.
First, we adopt a more intuitive but restrictive notation.
Define $\tf(x,\gamma)$ such that 
\[ \begin{split}
  \tf (f(x), \gamma) = f^\gamma(x).
\end{split} \]
In other words, we write $f^\gamma $ by appending a small perturbative map $\tf$ to $f:=f^{0}$.
Here, $\tf$ is the identity map when $\gamma=0$.
Note that $\tf$ can be defined only if $f$ satisfies 
\[ \begin{split}
  f^\gamma(x) = f^\gamma(x')
  \quad \textnormal{whenever} \quad 
  f(x) = f(x').
\end{split} \]
For example, when $f$ is bijective, then we can well-define $\tf$.
Hence, this new notation is more restrictive than what we used in other parts of this section, but it allows us to see more clearly what happens during the perturbation.

With the new notation, the dynamics can now be written as
\[ \begin{split}
  X_n =  \tf f(X_{n-1}) + Y_n ,
  \quad \textnormal{where} \quad 
  Y_n \iid  p.
\end{split} \]
By this notation, only $f$ changes time step, but $\tf$ and adding $Y$ do not change time.
In this subsection, we shall start from after having applied the map $f$, and only look at the effect of applying $\tf$ and adding noise:
this is enough to account for the essentials of \Cref{l:local1step} and \Cref{l:niu}.
Roughly speaking, the $q_n$ we use in the following is in fact $q_n:=L_f h_{n-1}$;
$q_n$ is the density of $z_n:=fx_{n-1}$, so $z_n+y_n=x_n$, and we omit the subscript for time $0$.

With the new notation, \Cref{l:local1step} is essentially equivalent to
\begin{equation} \begin{split} \label{e:dao}
  \delta (L_p L_{\tf} q) (x)
  = \int -  \delta \tf z \cdot dp(x - z) \, q(z) dz.
\end{split} \end{equation}
We explain this intuitively by \Cref{f:intui}.
Let $\tz:= \tf z$ be distributed according to $\tq$, then $L_p \tq$ is obtained by first attaching a density $p$ to each $\tz$, then integrating over all $\tz$.
$L_p L_\tf q$ is obtained by first moving $z$ to $\tz$ and then performing the same procedure.
Hence, $\delta L_p L_\tf q(x)$ is first computing $\delta p_{\tf z}(x)$ for each $z$, then integrating over $z$.
Let
\[ \begin{split}
   p_{\tf z}(x) := p(x - \tf z)
\end{split} \]
be the density of the noise centered at $\tf z$.
So
\[ \begin{split}
  \delta p_{\tf z}(x)
  = -dp_{\tf z}(x) \cdot \delta \tf z (x)
  = -dp(x - \tf z) \cdot \delta \tf z.
\end{split} \]
Here, $\delta\tf z (x)$ in the middle expression is the horizontal shift of $p_{\tf z}$'s level set previously located at $x$.
Since the entire Gaussian distribution is parallelly moved on the Euclidean space $\R^M$, so $\delta\tf z (x)=\delta\tf z$ is constant for all $x$.
Then we can integrate over $z$ to get \Cref{e:dao}.

\begin{figure}[ht] \centering
  \includegraphics[scale=0.4]{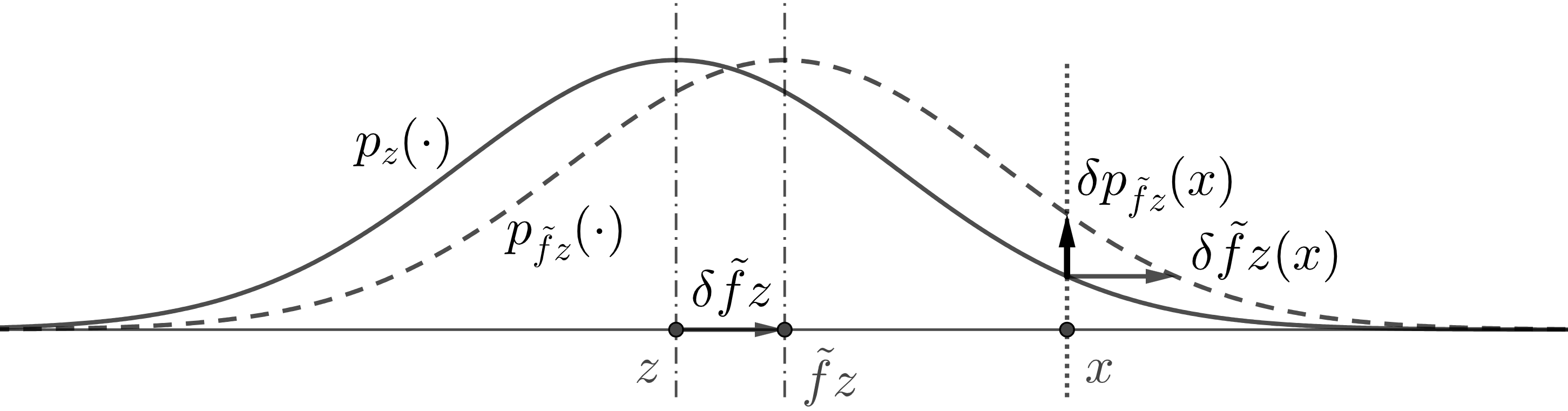}
  \caption{Intuitions for \Cref{l:local1step} and \Cref{l:niu}.}
  \label{f:intui}
\end{figure}

Using $\tf$ notation, \Cref{l:niu} is equivalent to
\[ \begin{split}
   \delta \int \Phi(x)  (L_p L_\tf q) (x) dx
  = \iint - \Phi(x) \delta \tf (z)\cdot \frac {dp(y)}{p(y)} \, q(z) dz \, p(y) dy,
\end{split} \]
where $x$ on the right is a function $x=z+y$.
Intuitively, this says that the left side equals to first compute
$\delta \int \Phi(z+y) p_{\tf z}(y) dy $
for each $z$, then integrate over $z$.
The integration on the left uses $x$ as the dummy variable, which is convenient for the transfer operator formula above.
But it does not involve a density for $x$, so it is not immediately ready for Monte-Carlo.
The right integration is over $q(z)$ and $p(y)$, which are easy to sample.

It is important to differentiate only $p$ but not $f$.
In fact, our core intuitions are completely at the same time point, and do not even involve $f$.
Hence, we can easily generalize to the cases where $f$ is bad, for example, when $f$ is not bijective or when $f$ is not hyperbolic: these are all difficult for deterministic systems.

\subsection{Kernel-differentiation over finitely many steps}
\label{s:manystep}

We use \Cref{l:niu} to get integrated formulas for $ h_T:=  ( L_p L_{f^\gamma})^T h_0$ and its perturbation $\delta h_T$, which can be sampled by Monte-Carlo type algorithms.
Note that here $h_0$ is fixed and does not depend on $\gamma$.
Here, the formula for $\delta h_T$ was previously known as the likelihood ratio method or the Monte-Carlo gradient method and was given in \cite{Rubinstein1989,Reiman1989,Glynn1990}.
The magnitude of the integrand increases as $O(\sqrt{T})$ after centralization, where $T$ is the number of time steps.

\begin{lemma} \label{l:PhiMC}
Under \Cref{ass:basic}, for any $C^2$ (not necessarily positive) densities $h_0$, any $n\in \N$,
  \[ \begin{split}
    \int \Phi(x_n) ((L_pL_f)^n h_0) (x_n) dx_n
  = \int \Phi(x_n) h_0(x_0) dx_0 p(y_1) dy_1 \cdots p(y_n) dy_n.
  \end{split} \]
Here the $x_n$ on the left is a dummy variable; whereas $x_n$ on the right is recursively defined by $x_m=f(x_{m-1})+y_m$, so $x_n$ is a function of the dummy variables $x_{0}, y_{1}, \cdots, y_{n}$.
\end{lemma}

\begin{remark*}
  The integrated formula is straightforward to compute by Monte-Carlo method.
  That is, for each $l$, we generate random $x_{l,0}$ according to density $h_0$, and $y_{l,1}, \cdots, y_{l,n}$ i.i.d. according to $p$.
  Then we compute $\Phi(x_n)$ for this particular sample orbit of $ \left\{x_{l,0}, y_{l,1}, \cdots, y_{l,n} \right\} $.
  Note that the experiments for different $l$ should be independent or decorrelated.
  Then the Monte-Carlo integration is simply
\[ \begin{split}
  \int \Phi(x_n) ((L_pL_f)^n h_0) (x_n) dx_n
  = \lim_{L\rightarrow\infty} \frac 1L \sum_{l=1}^L \Phi(x_{l,n}).
\end{split} \]
Almost surely according to $h_0\times p_1\times\cdots\times p_n$.
In this paper, we shall use $n, m$ to indicate time steps, whereas $l$ labels different samples.
\end{remark*}

\begin{proof}
Sequentially apply the definition of $L_p$ and $L_f$, we get
\begin{equation} \begin{split} \label{e:year}
    \int \Phi(x_n) ((L_pL_f)^n h_0) (x_n) dx_n \\
    = \iint \Phi(y_n+z_n) (L_f(L_pL_f)^{n-1} h_0) (z_n) p(y_n) dy_n dz_n \\
    = \iint \Phi(x_n(x_{n-1},y_n)) ((L_pL_f)^{n-1} h_0) (x_{n-1}) p(y_n) dy_n dx_{n-1}.
\end{split} \end{equation}
  Here, $x_n:=y_n+f(x_{n-1})$ in the last expression is a function of the dummy variables $x_{n-1}$ and $y_n$.
  Roughly speaking, the dummy variable $z_n$ in the second expression is $f(x_{n-1})$.
  
  Recursively apply \Cref{e:year} once, we have 
  \[ \begin{split}
    \int \left(\int \Phi(x_n(x_{n-1},y_n)) ((L_pL_f)^{n-1} h_0) (x_{n-1}) dx_{n-1} \right) p(y_n) dy_n\\
    =
    \int \Phi(x_n(x_{n-1}(x_{n-2},y_{n-1}),y_n)) ((L_pL_f)^{n-2} h_0) (x_{n-2}) p(y_{n-1})p(y_n) dy_{n-1} dy_n dx_{n-2}.
  \end{split} \]
  Here $x_{n-1}:=y_{n-1}+f(x_{n-2})$, so $x_n$ is a function of the dummy variables $x_{n-2},y_{n-1}$, and $y_n$.
  Keep applying \Cref{e:year} recursively to prove the lemma.
\end{proof}

\begin{theorem} [kernel-differentiation formula for $\delta h_T$] \label{l:finiteT}
Under \Cref{ass:basic}, then
\[ \begin{split}
  \delta \int \Phi(x_T) h_T(x_T) dx_T
  = - \int \Phi(x_T) \sum_{m=0} ^ {T-1} \left( \delta f^\gamma x_{m} \cdot \frac {dp}{p} (y_{m+1}) \right) h_0(x_0) dx_0 (pdy)_{1\sim T}.
\end{split} \]  
Here $dp$ is the differential of $p$, $(pdy)_{1\sim T}:=p(y_1)dy_1\cdots p(y_T)dy_T$ is the independent distribution of $Y_n$'s, and $x_m$ is a function of dummy variables $x_0, y_1, \cdots, y_m$.
\end{theorem}

\begin{remark*}
Here, the benefit of the likelihood ratio trick is more obvious: it recovers the measure under which we integrate to the original probability measure on orbits.  
\end{remark*}

\begin{proof}
 Note that $\Phi$ is fixed, use \Cref{e:xu}, and let $m=T-n-1$, we get
\[ \begin{split}
  \delta \int \Phi(x_T) h_T(x_T) dx_T
  = \int \Phi(x_T) (\delta h_T) (x_T) dx_T \\
  = \int \Phi(x_T) \sum_{m=0}^{T-1} \left( ( L_p L_f )^{T-m-1} \, \delta(L_p L_{f^\gamma} ) \, (L_p L_f )^{m} h_0 \right) (x_T) dx_T.
\end{split} \]

For each $m$, first apply \Cref{l:PhiMC} several times,
\[ \begin{split}
  \int \Phi(x_T) \left( ( L_p L_f )^{T-m-1} \, \delta(L_p L_{f^\gamma} ) \, (L_p L_f )^{m} h_0 \right) (x_T) dx_T \\
  = \int \Phi(x_T) \left( \delta(L_p L_{f^\gamma} ) \, (L_p L_f )^{m} h_0 \right) (x_{m+1}) \, dx_{m+1} (pdy)_{m+2\sim T} 
\end{split} \]
Here $x_T$ 
is a function of $x_{m+1}, y_{m+2}, \cdots, y_T$.
Then apply \Cref{l:niu} once,
\[ \begin{split}
  = - \int \Phi(x_T) \delta f^\gamma x_m \cdot \frac {dp}p (y_{m+1}) \left( (L_p L_f )^{m} h_0 \right) (x_{m}) \, dx_{m} (pdy)_{m+1\sim T} 
\end{split} \]
Now $x_T$ is a function of $x_{m}, y_{m+1}, \cdots, y_T$.
Then apply \Cref{l:PhiMC} several times again, 
\[ \begin{split}
  = - \int \Phi(x_T) \delta f^\gamma x_m \cdot \frac {dp}p (y_{m+1}) h_0 (x_0) \, dx_0 (pdy)_{1\sim T} 
\end{split} \]
Then sum over $m$ to prove the theorem.
\end{proof}

Note that subtracting any constant from $\Phi$ does not change the linear response: we prove a more detailed statement in \Cref{l:int0}.
Hence, we can centralize $\Phi$, that is, replacing $\Phi(\cdot)$ by $\Phi(\cdot)-\Phi_{avg,T}$, where the constant
\[ \begin{split}
  \Phi_{avg,T} := \int \Phi(x_T) h_T(x_T) dx_T.
\end{split} \]
We sometimes centralize by subtracting $\Phi_{avg}:= \int \Phi h dx$.
The centralization reduces the amplitude of the integrand, so the Monte-Carlo method converges faster:
this is also a known result and was discussed, for example, in \cite{GlOl19}.

\begin{lemma}[free centralization]
\label{l:int0}
Under \Cref{ass:basic}, for any $m\ge0$, we have
\[ \begin{split}
  \int  \delta f^\gamma x_{m} \cdot \frac {dp}{p} (y_{m+1})  h_0(x_0) dx_0 (pdy)_{1\sim m+1}
  =
  0.
\end{split} \] 
\end{lemma}

\begin{proof}
First note that a density function $p$ on $\R^M$ with a bounded $C^2$ norm must have $\lim_{y\rightarrow\infty}p(y)=0$, since otherwise its integration cannot be one.

Then notice that $x_m$ is a function of dummy variables $x_0, y_1, \cdots, y_m$, so we can first integrate
\[ \begin{split}
  \delta f^\gamma x_{m} \cdot \int \frac {dp}{p} (y_{m+1})  p(y_{m+1}) dy_{m+1}
  = \delta f^\gamma x_{m} \cdot \int dp (y_{m+1}) dy_{m+1}
  = 0.
\end{split} \]
Since $p(y)\rightarrow 0$ as $y\rightarrow\infty$.
Here, $dp$ is the differential of the function $p$, whereas $dy$ indicates the integration.
\end{proof}

For convenience in computer coding, we explicitly rewrite this theorem in a centralized and time-inhomogeneous form, where $\Phi$, $f$, and $p$ are different for each step.
This is the setting for many applications, such as finitely deep neural networks.
The following proposition can be proved similarly to our previous proofs.
Note that we can reuse a lot of data should we also care about the perturbation of the averaged $\Phi_n$ for other layers $1\le n\le T$.

In the following proposition, let $\Phi_T:\R^{M_T}\rightarrow\R$ be the observable function defined on the last layer of dynamics.
Let $h_T$ be the pushforward measure given by the dynamics in \Cref{e:dynamicinhomo}, that is, defined recursively by $h_n:=L_{p_n} L_{f^\gamma_{n-1}} h_{n-1}$;
Let $(pdy)_{1\sim T}:=p_1(y_1)dy_1\cdots p_T(y_T)dy_T$ be the independent but not necessarily identical distribution of $Y_n$'s.

\begin{proposition} [centralized kernel-differentiation formula for $\delta h_T$ of time-inhomogeneous systems] \label{l:dhTn}
If $f_{\gamma,n}$ and $p_n$ satisfy \Cref{ass:basic} for all $n$, then
\[ \begin{split}
  &\delta \int \Phi_T(x_T) h_T(x_T) dx_T
  \\
  = &- \int \left(\Phi_T(x_T) - \Phi_{avg,T} \right) \sum_{m=0} ^ {T-1} \left( \delta f_{\gamma,m} x_{m} \cdot \frac {dp}{p} (y_{m+1}) \right) h_0(x_0) dx_0 (pdy)_{1\sim T}.
\end{split} \]  
Here $\Phi_{avg,T} := \int \Phi_T h_T $.
\end{proposition}

\subsection{Kernel-differentiation over infinitely many steps} 
\label{s:infinite}

For the perturbation of physical measures, the kernel-differentiation formula for $\delta h$ can take the form of a long-time average on an orbit.
This reduces the magnitude of the integrand, so the convergence is faster.
Note that this subsection and the next require that the dynamics is time-homogeneous ($f$ and $p$ are the same for each step) for the existence of physical measures and for invoking the ergodic theorem.

\begin{lemma}[kernel-differentiation formula for physical measures]
\label{l:kd4h}
Under \Cref{ass:basic,ass:hdh}, 
\[ \begin{split}
\delta \int \Phi(x) h(x) dx
= \int \Phi(x) \delta h (x) dx
\\
= - \lim_{W\rightarrow\infty} \sum_{n=1} ^ {W} \int \Phi(x_{n}) \left( \delta f^\gamma x_{0} \cdot \frac {dp}{p} (y_{1}) \right) h^{ 0}(x_0) dx_0 (pdy)_{1\sim n}
\end{split} \]
\end{lemma}

\begin{proof}
By \Cref{ass:hdh}, we can start from initial density $h_0 = h^0$, apply the perturbed dynamics $f^{\gamma}$ many times, and 
$h_T:= ( L_p L_{f^\gamma})^T h^{ 0}$ would converge to $h^\gamma$.

By \Cref{l:finiteT}, for any $T$,
\[ \begin{split}
  \delta \int \Phi(x_T) h_T(x_T) dx_T
  = -  \sum_{m=0} ^ {T-1} \int \Phi(x_T) \left( \delta f^\gamma x_{m} \cdot \frac {dp}{p} (y_{m+1}) \right) h^{0}(x_0) dx_0 (pdy)_{1\sim T}.
\end{split} \]
Note that $h^{0}$ is the invariant measure for $\gamma=0$, so the expression
\[ \begin{split}
  = -  \sum_{m=0} ^ {T-1} \int \Phi(x_{T-m}) \left( \delta f^\gamma x_{0} \cdot \frac {dp}{p} (y_{1}) \right) h^{0}(x_0) dx_0 (pdy)_{1\sim T-m}
\end{split} \]
Passing $T-m$ to $n$, 
\[ \begin{split}
  = - \sum_{n=1} ^ {T} \int \Phi(x_{n}) \left( \delta f^\gamma x_{0} \cdot \frac {dp}{p} (y_{1}) \right) h^0(x_0) dx_0 (pdy)_{1\sim n}
\end{split} \]
Note that this is the correlation function between $\Phi(x_{n})$ and $\delta f^\gamma x_{0} \cdot \frac {dp}{p} (y_{1})$.
By \Cref{l:int0}, $ \E[\delta f^\gamma X_{0} \cdot \frac {dp}{p} (Y_{1})]=0$.
By \Cref{ass:hdh}, the derivatives expressed by the above summations uniformly converge to the expression in the lemma.
Hence, the limit equals $\delta h$.
\end{proof}

\ree{
Compared to \cite{GG19}, \Cref{l:kd4h} added a `likelihood ratio' trick (or rather an observation): divide and multiply by $p(y_{1})$.
Hence, the probability under which our formula is integrated, $h^{ 0}(x_0) dx_0 (pdy)_{1\sim n}$, is the same as the probability of sample paths of the original system.
This is convenient for Monte-Carlo type numerics, since we only need to compute certain expression (the integrand) on each sample path, then take average.
}

\lil{sun}
\ree{
More specifically, to generate a sample of the integrand, 
$\Phi(x_{n}) \left( \delta f^\gamma x_{0} \cdot \frac {dp}{p} (y_{1}) \right)$,
we first generate sample $X_0, Y_1, Y_2,\cdots \sim h\times p\times p \times \cdots$, then compute $X_1, \ldots X_n$ according to the dynamic, and then compute a sample of the integrand.
Thus, we have one sample of the integrand distributed according to the probability measure in our linear response formula.
Monte-Carlo only means to compute many samples and take the average.
}

\ree{
Note that we need to know expression of the `likelihood ratio' of the kernel, $dp_n/p_n$, to compute the integrand.
For Gaussian, this expression is in \Cref{e:ttbm}.
On the other hand, $h_n$ is not differentiated and can be sampled by the dynamics, so we do not need its expression.
Indeed, differentiating $h_n$ is difficult, since it must be influenced by the dynamics, so its derivative should involve information from even earlier steps.
For example, the equivariant divergence formula \cite{TrsfOprt,vdivF,fr} gives an expression of the derivative of $h_n$, which depends on quantities from previous time steps.
But the benefit of the kernel-differentiation method is that $\delta Lh$ can be resolved in one step, and the derivative hits $p_n$ but not $h_n$.
}

\subsection{Ergodic kernel-differentiation formula} 
\label{s:ergodic}

We can apply the ergodic theorem, forget the details of the initial distribution, and sample $\delta h$ by one long orbit.

\goldbach*

\begin{remark*}
    Recall that $W$ is determined by the rate of decay of correlations.
    $L$ is the total number of steps.
    Typically $W\ll L$ in numerics.
\end{remark*}

\begin{proof}

  Since $h$ is invariant ergodic for the dynamic $X_{n+1}=f(X_n)+Y_{n+1}$, if $X_0\sim h$, then $\{X_l, Y_{l+1},\cdots,Y_{l+n}\}_{l\ge0}$ is a stationary process, so we can apply Birkhoff's ergodic theorem (the version for stationary processes), 
\[ \begin{split}
\int \Phi(x_{n}) \left( \delta f^\gamma x_{0} \cdot \frac {dp}{p} (y_{1}) \right) h^0(x_0) dx_0 (pdy)_{1\sim n}
\\ \overset{\mathrm{a.s.}}{=}
\lim_{L\rightarrow\infty} \frac 1L \sum_{l=1}^L \Phi(X_{n+l}) \, \delta f^\gamma X_l \cdot \frac{dp}p(Y_{1+l}) .
\end{split} \]
  By substitution we have
  \[ \begin{split}
    \delta \int \Phi(x) h(x) dx
    \,\overset{\mathrm{a.s.}}{=}\, 
    \lim_{W\rightarrow\infty} \lim_{L\rightarrow\infty} - \sum_{n=1}^W \frac 1L \sum_{l=1}^L \Phi(X_{n+l}) \, \delta f^\gamma X_l\cdot \frac{dp}p(Y_{1+l})
  \end{split} \]
  Then we can rerun the proof after centralizing $\Phi$.
\end{proof}

\lil{ma}
\ree{
With our results worked out, in the rest of this subsection, we explain its improvement over the previous likelihood ratio method \cite{Rubinstein1989,Reiman1989,Glynn1990,PSW23}.
}
The previous likelihood ratio method computes the linear response of physical measure by 
\[ \begin{split}
\delta \int \Phi(x) h(x) dx
= - \lim_{T\rightarrow\infty} \sum_{n=0} ^ {T-1} \int \Phi(x_{T}) \left( \delta f^\gamma x_{n} \cdot \frac {dp}{p} (y_{n+1}) \right) h_0 (x_0) dx_0 (pdy)_{1\sim n}
\end{split} \]
Here $h_0$ is any density function.
Their result differs from \Cref{l:kd4h}, which uses the physical measure $h$ for $\gamma=0$.
This seemingly subtle difference causes a problem for Monte-Carlo algorithm: Since $h_0$ is not stationary, they can not use the ergodic theorem as we did in \Cref{t:dh}.
Rather, the previous likelihood ratio method uses
\[ \begin{split}
\delta \int \Phi(x) h(x) dx
\approx
- \frac 1L \sum_{l=1} ^ {L} \sum_{n=0} ^ {T-1} \Phi(X_{l,T}) \left( \delta f^\gamma X_{l,n} \cdot \frac {dp}{p} (Y_{l,n+1}) \right)
\end{split} \]
For some large $L$.
Here $l\in[1,L]$ labels the different sample paths, and $X_0\sim h_0$.

The problem of using their expression is that we need $T$ to be larger than some $T'$ so that $h_T\approx h$.
Roughly speaking, for each path of length $T$, only at the end of the path is the statistic useful.
So we compute $T$ many steps to obtain \textit{one} data for the average over $L$.
In comparison, in our \Cref{t:dh}, we obtain one data for each step we compute, since we only have one path from the stationary distribution.

Moreover, we need $T$ to be larger than some $W'$ so that the decorrelation between $\Phi(x_T)$ and $\delta f^\gamma x_{n} \cdot \frac {dp}{p} (y_{n+1})$ takes effect.
When $T' > W'$, then the integrand in the likelihood ratio method involves the sum of $T>T'>W'$ terms.
In contrast, in our \Cref{t:dh}, we only require $T> W'$. 
Hence, the size of our integrand is smaller, so we need fewer samples to compute the expectation.

In summary, we generate samples more efficiently and sometimes require fewer samples than the previous likelihood ratio method.

\subsection{Further generalizations}
\label{s:generalize}

\subsubsection{\texorpdfstring{$p$}{p} depends on \texorpdfstring{$z$}{z} and \texorpdfstring{$\gamma$}{gamma}}
\hfill\vspace{0.1in}
\label{s:pgammaz}

Our pictorial intuition does not care whether $p$ depends on $\gamma$ or $z$. So we can generalize all of our results to the case
\[ \begin{split}
  X_{n+1} = f(X_n) +Y_{n+1},
  \quad \textnormal{where} \quad 
  \Pro(Y_{n+1}=y \, | \, f(X_n)=z) = p^\gamma_{z}(y).
\end{split} \]
We write these formulas without repeating the proofs.

\Cref{e:Lppt} and \Cref{e:Lpint} become
\[ \begin{split}
  L_{p^\gamma} q (x) 
  = \int q(x-y) p^\gamma_{x-y}(y) dy
  = \int q(z) p^\gamma_{z}(x-z) dz.
  \\
  \int \Phi (x) L_{p^\gamma} q (x) dx
  = \iint \Phi(y+z) q(z) p^\gamma_{z}(y) dz dy.
\end{split} \]
\Cref{l:PhiMC} becomes
  \[ \begin{split}
    \int \Phi(x_n) ((L_{p^\gamma} L_{f^\gamma})^n h_0) (x_n) dx_n
  = \int \Phi(x_n) h_0(x_0) dx_0 p^\gamma_{z_1}(y_1) dy_1 \cdots p^\gamma_{z_n}(y_n) dy_n.
  \end{split} \]
Here $z_m$ and $x_m$ on the right are recursively defined by $z_m=f(x_{m-1})$, $x_m=z_m + y_m$.
We also have the pointwise formula
\[ \begin{split}
  (L_{p^\gamma} L_{f^\gamma} h) (x_1)
  = \int h(x_0) p^\gamma_{f^\gamma x_0}(x_1-f^\gamma x_0) dx_0.
\end{split} \]

Since now $p$ depends on $\gamma$ via three ways, \Cref{l:local1step} becomes
\[ \begin{split}
  \delta (L_{p^\gamma} L_{f^\gamma} h) (x_1)
  = \int \dd p\gamma (x_0,y_1) h(x_0) dx_0,
  \quad \textnormal{where} \quad 
  z_1 = fx_0,
  \quad
  y_1 = x_1-z_1,
  \\
  \dd p\gamma (x_0,y_1) 
  := \dd {}\gamma p^\gamma_{ f^\gamma x_0} (x_1-f^\gamma x_0)
  = \delta p (z_1, y_1) + \delta f^\gamma (x_0)\cdot \left(\pp p z -  \pp p y\right) (z_1,y_1) .
\end{split} \]
Here derivatives $\pp p z$ and $\pp p y$ refer to writing the density as $p^\gamma_{z}(y)$.
If $p$ does not depend on $\gamma$ and $z$, then we recover \Cref{l:local1step}.
\Cref{l:niu} becomes
\[ \begin{split}
  \delta \int \Phi(x_1) (L_{p^\gamma} L_{f^\gamma} h) (x_1) dx_1
  = \int \Phi(x_1) \delta (L_{p^\gamma} L_{f^\gamma} h) (x_1) dx_1
  \\
  = \iint \Phi(x_1) 
  \frac {dp}{p d\gamma} (x_0,y_1) h(x_0) dx_0 \, p_{z_1}(y_1) dy_1.
\end{split} \]
Here 
\[ \begin{split}
  \frac {dp}{p d\gamma} (x_m,y_{m+1})
  := \frac 1 { p_{f x_m}(y_{m+1}) }
  \dd {}\gamma p^\gamma_{ f^\gamma x_m} (x_{m+1}-f^\gamma x_m).
\end{split} \]
\Cref{l:finiteT} becomes
\[ \begin{split}
  \delta \int \Phi(x_T) h_T(x_T) dx_T
  = \int \Phi(x_T) \sum_{m=0} ^ {T-1} 
  \frac {dp}{p d\gamma} (x_m,y_{m+1}) 
  h_0(x_0) dx_0 (p_z dy)_{1\sim T}.
\end{split} \]  
Here $(p_zdy)_{1\sim T}:=p_{z_1}(y_1)dy_1\cdots p_{z_T}(y_T)dy_T$ is the distribution of $Y_n$'s;
$z_{m+1}=fx_m$ is a function of dummy variables $x_0, y_1, \cdots, y_m$.

We still have free centralization for $\Phi$.
So \Cref{l:dhTn} and \Cref{t:dh} still hold accordingly, in particular, when the system is time-homogeneous and has \re{physical density} $h$, we still have
\[ \begin{split}
  \delta \int \Phi(x) h(x) dx
  \,\overset{\mathrm{a.s.}}{=}\, 
  \lim_{W\rightarrow\infty} \lim_{L\rightarrow\infty} \frac 1L \sum_{n=1}^W
  \sum_{l=1}^L \left( \Phi(X_{n+l}) - \Phi_{avg} \right) \,
  \frac{dp}{pd\gamma}(X_l,Y_{1+l}).
\end{split} \]

\subsubsection{General random dynamical systems}
\hfill\vspace{0.1in}
\label{s:rand}

For general random dynamical systems, at each step $n$, we randomly select a map $g$ from a family of maps, denote this random map by $G$, the dynamics is now
\[ \begin{split}
  X_{n+1} = G(X_n).
\end{split} \]
The selections of $G$'s are independent among different $n$.
It is not hard to see that our pictorial intuition still applies, so our work can generalize to this case.

We can also formally transform this general case into the case in the previous subsubsection.
To do so, define the deterministic map $f$
\[ \begin{split}
  f(x) := \E [G(x)],
\end{split} \]
where the expectation is with respect to the randomness of $G$.
Then the dynamic equals in distribution to 
\[ \begin{split}
  X_{n+1} = f(X_n) + Y_{n+1},
  \quad \textnormal{where} \quad 
  Y_{n+1} = G(X_n) - f(X_n).
\end{split} \]
Hence, the distribution of $Y_{n+1}$ is completely determined by $X_n$ and the distribution of $G$.
Note that we still need to compute derivatives of the distribution of $Y$, or $G$.


\section{Foliated noise and perturbation}
\label{s:foliation}

This section shows that the kernel-differentiation formulas are still correct for the case where the noise and $\delta f$ are along a given foliation.

\subsection{Preparation: measure notations, foliations, conditional densities}
\label{s:prep2}

\lil{kou}
\ree{
In this section we consider a more general phase space, $\cM$, which is a $M$-dimensional manifold.
}
We assume that a neighborhood of the support of the physical measure $\mu$ can be foliated by a single chart: this basically means that we can partition the neighborhood by $c$-dimensional submanifolds.
More specifically, there is $F = \{F_\alpha\}_{\alpha\in A}$, a family of $c$-dimensional $C^3$ submanifolds on $\cM$, and a $C^3$ diffeomorphism between the neighborhood and an open set in $\R^M$, such that the image of $F_\alpha$ is the $c$-dimensional horizontal plane, $\{x:x_1=\ldots =x_{M-c} = 0\}$.
Each submanifold is called a `leaf'.
An extension to the case where the neighborhood admits not a single chart but a foliated atlas composed of many charts, is possible, and we leave it to the future.

Under single-chart foliation, this section considers the system
\[
Z_n = f(X_n)
\,,\quad
\tZ_n = \tf(Z_n)
\,,\quad
X_{n+1} = X_{n+1} (\gamma, \tZ_n).
\]
Here $\tf$ depends on $\gamma$ with default value $\gamma=0$, $f$ is fixed and independent of $\gamma$.
The last equation just means that $X_{n+1}$ is a random variable whose distribution depends on $\gamma$ and $\tZ_n$, with details to be explained in the following paragraphs.
The notation in the first equation relates to our previous notation by $f^\gamma = \tf \circ f$.

We assume that $X$ and $\tf$ are parallel to $F_\alpha$.
More specifically, for any $\gamma$ in a small interval on $\R$ and any $z\in \cM$,
\[
\tz = \tf(z) \in  F_\alpha(z)
\;;\quad
X(\gamma, \tz) \in F_\alpha(z).
\]
where $F_\alpha(z)$ is the leaf $F_\alpha$ at $z$.
Note that we do \textit{not} require $f$ to be constrained by the foliation, so $x$ and $fx$ can be on different leaves: This is more general than systems restricted on one submanifold.

The measures considered in this section are not necessarily absolute continuous with respect to Lebesgue on $\cM$.
So we should use another set of notation for measures.
The transfer operator on measures is denoted by $f_*$, more specifically, for any measure $\nu$,
\[\begin{split}
  \int \Phi(x) \, d (f_* \nu)(x)
  := 
  \int \Phi(fx)  d\nu (x).
\end{split}\]

We can disintegrate a measure $\nu$ into conditional measure $[\nu]^F_\alpha$ on each leaf and the quotient measure $\nu'$ on the set of $\alpha$.
Since we have only one fixed foliation, we shall just omit the superscript $F$.
Here $\{[\nu]_\alpha\}_{\alpha\in A}$ is a family of probability measures such that each `lives on' $F_\alpha$.
Moreover, the integration of any smooth observable $\phi$ can be written in two consecutive integrations, 
\[
\int_\cM \phi(z) d\nu(z)
= \int_A \int_{F_\alpha} \phi(z) d[\nu]_\alpha(z) d\nu'(\alpha).
\]
We typically use $\nu$ to denote the measure of $Z$.

We assume that the random variable $X(\gamma,\tz)$ has a $C^3$ density function $p(\tz, \cdot)$ with respect to the Lebesgue measure on $F_\alpha(z)$.
That is, $p$ is the conditional density of $X|\tZ$.
Then we can define the operator $p_*$, which transfers $\tnu$, the measure of $\tZ$, to $\mu$, the measure of $X$.
It transports measures only within each leaf, but not across different leaves, so the quotient measure $\tnu'$ is left unchanged.
On the other hand, the density of the conditional measure is changed to
\begin{equation}    \label{e:pstar}
[p_* \tnu]^{\textrm{dst}}_\alpha (x)
:= \dd{[p_* \tnu]_\alpha}{m} (x)
= \int_{F_\alpha} p(\tz, x) d[\tnu]_\alpha (\tz),
\end{equation}
where $d(\cdot)/dm$ is the Radon–Nikodym derivative with respect to the Lebesgue measure on $F_\alpha$.
Note that $[\tnu]_\alpha$ does not necessarily have a density on $F_\alpha$.
We have the double and triple integration formulas for the expectation of joint distribution 
\begin{equation}\begin{split}  \label{e:23}
\E [\phi(X,\tZ)]
= \int_{\tz\in \cM} \int_{x\in F_\alpha(\tz)}  \phi(x,\tz) p(\tz, x) dx d \tnu (\tz)
\\
= \int_{\alpha\in A} \int_{x\in F_\alpha} \int_{\tz \in F_\alpha} \phi(x,\tz) p(\tz, x)  d[\tnu]_\alpha (\tz) dx d[\tnu]'(\alpha)
\end{split}\end{equation}
To summarize, we shall assume the following in this section.

\begin{assumption}
\label{ass:folibasic}
Assume that a neighborhood of the support of the measures of interest is foliated by a $C^3$ family of $C^3$ leaves $\{F_\alpha\}_{\alpha\in A}$.
For any $z\in\cM$, any $\gamma$ in a small interval in $\R$ around zero,
we have $\tf(z) \in  F_\alpha(z)$,
$X(\gamma, \tz) \in F_\alpha(\tz)$.
The density $p$ is $C^2$ on $F_\alpha$ (but is singular with respect to $\cM$). 
\end{assumption}


For the case of multiple steps, we shall denote the sequence of measures of $X_n$ by
\[
\mu_n^\gamma := (p_* \tf_* f_*)^n \mu_0
\;,\quad
\mu^\gamma := \lim_{n\rightarrow\infty} \mu_n,
\]
where the limit is in the weak* sense, and $\mu^\gamma$ is called the physical measure.
We still make \Cref{ass:hdh}, where statement (a) should be interpreted as the weak* convergence of the measure, not density, since the physical measure can be singular and does not have a density function.
Also, for statement (b), in the definition of correlation in \Cref{e:corre}, the density $p$ is now the conditional density.
For foliated cases, it is more difficult to prove \Cref{ass:hdh} from basic assumptions, although they seem true in many numerical applications; we did not find accurate references for this purpose.

\subsection{Foliated kernel-differentiation in one-step}
\label{s:foli}

We can see that the pictorial intuition in \Cref{s:intui} still makes sense on individual leaves.
So, it is not surprising that the main theorems are still effective.
We shall first prove the formula on one leaf in one-step, which is the foliated version of \Cref{l:local1step,l:niu}.
Note that here one-step means to transfer the measure by $\tf$, so the foliation is unchanged.

\begin{lemma} [local one-step kernel-differentiation on a leaf, comparable to \Cref{l:local1step}]
\label{t:foli1loc}
Under \Cref{ass:basic,ass:folibasic}, for any measure $\nu$,
\[
\delta [p_* \tf_* \nu]^{\textrm{dst}}_\alpha (x)
= \int_{F_\alpha} \delta \tf(z) \cdot d_z p(z, x) d[\nu]_\alpha (z)
\]
Here $\delta:=\partial/\partial \gamma|_{\gamma=0}$, 
and the integrand is the derivative of $p(z, x)$ with respect to $z$ in the direction $\delta \tf(z)$.
\end{lemma}

\begin{remark*}
    Note that $p(z,x)$ here becomes $p(x-z)$ in \Cref{l:local1step}, so taking derivatives in $z$ gives a minus sign.
    Compared to previous results, note that here $p(z, x)$ depends on $z$ in two ways: the distance between $z$ and $x$ could change if $z$ changes; the distribution centered at different $z$ may be different.
    Hence, roughly speaking, the term $d_z p$ here is the term $\pp p z -  \pp p y$ in \Cref{s:pgammaz}.
\end{remark*}

\begin{proof}
By \Cref{e:pstar},
\begin{equation} 
\label{e:dada}
[p_* \tf_* \nu]^{\textrm{dst}}_\alpha (x)
= \int_{F_\alpha} p(\tz, x) d[\tf_* \nu]_\alpha (\tz),    
\end{equation}

We give a formula for $d[\tf_* \nu]_\alpha$.
By definition of $p_*$ and conditional densities, for any smooth observable function $\phi$,
\[\begin{split} 
\int_{A} \int_{F_\alpha} \phi(\tz) d [\tf_* \nu]_\alpha (\tz) d[\tf_*\nu]'(\alpha)
= \int_{\cM} \phi(\tz) d \tf_* \nu (\tz)
\\
= \int_{\cM} \phi(\tf z) d \nu (z)
= \int_{A} \int_{F_\alpha} \phi(\tf z) d [\nu]_\alpha (z) d[\nu]'(\alpha).
\end{split}\]
Since $\tf_*$ preserves measure within each leaf, so the quotient measure $\nu'=[\tf_* \nu]'$ is unchanged.
Hence,
\[
\int_{F_\alpha} \phi(\tz) d [\tf_* \nu]_\alpha (\tz)
= \int_{F_\alpha} \phi(\tf z) d [\nu]_\alpha (z).
\]
In short, we find that pushing-forward by $\tf$ commutes with taking conditions.

Substituting back to \Cref{e:dada}, we get
\[
[p_* \tf_* \nu]^{\textrm{dst}}_\alpha (x)
= \int_{F_\alpha} p(\tz, x) d[\tf_* \nu]_\alpha (\tz)
= \int_{F_\alpha} p(\tf z, x) d[\nu]_\alpha (z)
\]
Then we can differentiate by $\gamma$ and evaluate at $\gamma=0$ to prove the lemma; 
note that $\tz =z$ at $\gamma=0$.
\end{proof}

\begin{lemma} [integrated one-step foliated kernel-differentiation, comparable to \Cref{l:niu}]
\label{t:foli1int}
Under \Cref{ass:basic,ass:folibasic}, for any fixed measure $\nu$ and any smooth observable $\Phi$,
\[ \begin{split}
\delta \int \Phi(x) d(p_*\tf_* \nu) (x) 
= \int_{z\in \cM} \int_{x\in F_\alpha(z)}  \Phi(x) \delta \tf(z) \cdot 
\frac{d_z p(z, x)}{ p(z, x)}  p(z, x) dx d\nu(z).
\end{split} \]
\end{lemma}

\begin{remark*}
  See \Cref{e:23} for the expectation notation which might help in understanding.
\end{remark*}

\begin{proof}

By the definition of conditional measure,
\[ \begin{split}
  \int \Phi(x) d(p_* \tf_* \nu) (x) 
  =  \int_A \int_{F_\alpha} \Phi(x) d[p_*\tf_*\nu]_\alpha(x) d[p_* \tf_*\nu]'(\alpha)
\end{split} \]
Since both $p_*$ and $\tf_*$ preserve the measure within each leaf, $[p_* \tf_*\nu]' = [\nu]'$ is not affected by $\gamma$.
Further express the conditional measure by integrating the density, we get
\begin{equation}
\label{e:lala}
\int \Phi(x) d(p_* \tf_*\nu) (x) 
= \int_A \int_{F_\alpha} \Phi(x) [p_*\tf_*\nu]^{\textrm{dst}}_\alpha(x) dx d[\nu]'(\alpha) 
\end{equation} 

Differentiate on both sides 
\[ \begin{split}
  \delta \int \Phi(x) d(p_*\tf_* \nu) (x) 
  = \delta \int_A \int_{F_\alpha} \Phi(x) [p_*\tf_* \nu]^{\textrm{dst}}_\alpha(x) dx d[\nu]'(\alpha)\\
\end{split} \]
Since every function in the integrand is bounded, we can move $\delta$ inside,
\[ \begin{split}
  = \int_A \int_{x\in F_\alpha} \Phi(x) \delta [p_*\tf_* \nu]^{\textrm{dst}}_\alpha(x) dx d[\nu]'(\alpha).
\end{split} \]
Apply \Cref{t:foli1loc}, 
\[ \begin{split}
  = \int_{\alpha\in A} \int_{x\in F_\alpha} \int_{z\in F_\alpha} \Phi(x) \delta \tf(z) \cdot d_z p(z, x) d[\nu]_\alpha (z) dx d[\nu]'(\alpha).
\\
  = \int_{\alpha\in A} \int_{x\in F_\alpha} \int_{z\in F_\alpha} \Phi(x) \delta \tf(z) \cdot 
  \frac{d_z p(z, x)}{ p(z, x)} d[\nu]_\alpha (z)  p(z, x) dx d[\nu]'(\alpha).
\end{split} \]
Then use \Cref{e:23} to transform the triple integration into double integration.
\end{proof}

\subsection{Foliated kernel-differentiation over finitely many steps}
\label{s:foliN}

This subsection proves results for finitely many steps.
Most proofs are similar to their counterparts in \Cref{s:manystep}, so we skip the proofs that are repetitive.

\begin{lemma}[comparable to \Cref{l:PhiMC}] \label{l:foliPhiMC}
Under \Cref{ass:basic,ass:folibasic}, let $\tz_m := \tf z_m := \tf f x_{m-1}$, then for any $\mu_0$, any $n\in \N$,
\[ \begin{split}
\int \Phi(x_n) d((p_*\tf_* f_*)^n \mu_0) (x_n)
\\= \int_{x_0\in \cM} \int_{x_1\in F_\alpha(\tz_1)} \cdots \int_{x_n\in F_\alpha(\tz_n)}
\Phi(x_n) p(\tz_n, x_n)dx_n \cdots p(\tz_1, x_1) dx_1 d\mu_0(x_0).
\end{split} \]
\end{lemma}

\begin{remark*}
    Note that for the foliated case, the integration domain of $x_n$ depends on $x_{n-1}$, so the multiple integration should first integrate $x_n$ with $x_{n-1}$ fixed.
\end{remark*}

\begin{theorem} [foliated kernel-differentiation formula for $\delta \mu_T$, comparable to \Cref{l:finiteT}] \label{l:foliT}
Under \Cref{ass:basic,ass:folibasic}, then
\[ \begin{split}
\delta \int \Phi(x_T) d\mu_T (x_T)
= \int_{x_0\in \cM} \int_{x_1\in F_\alpha(z_1)} \cdots \int_{x_T\in F_\alpha(z_T)}
\Phi(x_T) 
\\
\sum_{m=1} ^ {T} \left( \delta \tf(z_m) \cdot 
\frac{d_z p(z_m, x_m)}{ p(z_m, x_m)} \right)
p(z_T, x_T)dx_T \cdots p(z_1, x_1) dx_1 d\mu_0(x_0).  
\end{split} \]  
\end{theorem}

\begin{proof}
We shall recursively apply \Cref{t:foli1int}.
Note that $\nu$ is fixed there, but here $\mu_{T-1}$ is also affected by $\gamma$.
First, we differentiate the last step; by \Cref{t:foli1int},
\[ \begin{split}
\delta \int \Phi(x_T) d(p_*\tf_*f_* \mu_{T-1}) (x_T) 
\\
= \int_{x_{T-1}\in \cM} \int_{x_T\in F_\alpha(z_T)}  \Phi(x_T) \delta \tf(z_T) \cdot 
\frac{d_z p(z_T, x_T)}{ p(z_T, x_T)}  p(z_T, x_T) dx_T d\mu_{T-1}(x_{T-1})
\\
+ \int_{x_{T-1}\in \cM} \int_{x_T\in F_\alpha(z_T)}  \Phi(x_T) 
p(z_T, x_T) dx_T d \delta \mu_{T-1}(x_{T-1}).
\end{split} \]
Then we can apply \Cref{t:foli1int,e:pstar} to the last term repeatedly until $x_0$.
\end{proof}

\begin{lemma}[free centralization, comparable to \Cref{l:int0}]
\label{l:int0foli}
Under \Cref{ass:basic,ass:folibasic}, for any $1\le m\le T$, we have
\[ \begin{split}
\int_{x_0\in \cM} \int_{x_1\in F_\alpha(z_1)} \cdots \int_{x_T\in F_\alpha(z_T)}
\left( \delta \tf(z_m) \cdot 
\frac{d_z p(z_m, x_m)}{ p(z_m, x_m)} \right)
p(z_T, x_T)dx_T \cdots p(z_1, x_1) dx_1 d\mu_0(x_0)
=0.  
\end{split} \]  
\end{lemma}

\begin{proof}
We first do the inner integrations for steps later than $m$, which are all one.
Hence, the left side of the equation,
\[ \begin{split}
=
\int_{x_0\in \cM} \int_{x_1\in F_\alpha(z_1)} \cdots \int_{x_m\in F_\alpha(z_m)}
\delta \tf(z_m) \cdot 
\frac{d_z p(z_m, x_m)}{ p(z_m, x_m)}
p(z_m, x_m)dx_m \cdots p(z_1, x_1) dx_1 d\mu_0(x_0)
\\
=
\int_{x_{m-1}\in \cM} \int_{x_m\in F_\alpha(z_m)}
\delta \tf(z_m) \cdot \frac{d_z p(z_m, x_m)}{ p(z_m, x_m)} 
p(z_m, x_m)dx_m d\mu_{m-1}(x_{m-1}).
\end{split} \]
Here, $\mu_{m-1}$ is evaluated at $\gamma=0$ by default.
Then we can use \Cref{t:foli1int} with $\Phi$ set as a constant to prove the lemma.
\end{proof}

\silverbach*

\subsection{Foliated ergodic kernel-differentiation formula over infinitely many steps}
\label{s:foliinfi}

This subsection is comparable to \Cref{s:infinite,s:ergodic} and states results for infinitely many steps.
We shall skip the proofs since the ideas are the same and only the notation is a bit different.


\begin{lemma}[foliated kernel-differentiation formula for physical measures, comparable to \Cref{l:kd4h}]
Under \Cref{ass:basic,ass:hdh,ass:folibasic},
\[ \begin{split}
\delta \int \Phi(x) d\mu (x)
= \lim_{W\rightarrow\infty} \sum_{n=1} ^ {W} \int_{x_0\in \cM} \int_{x_1\in F_\alpha(z_1)} \cdots \int_{x_n\in F_\alpha(z_n)}
\Phi(x_n) 
\\
\left( \delta \tf(z_1) \cdot 
\frac{d_z p(z_1, x_1)}{ p(z_1, x_1)} \right)
p(z_n, x_n)dx_n \cdots p(z_1, x_1) dx_1 d\mu(x_0).  
\end{split} \]  
\end{lemma}

\begin{theorem} [orbitwise foliated kernel-differentiation formula for $\delta h$, comparable to \Cref{t:dh}]
\label{t:foli}
Under \Cref{ass:basic,ass:hdh,ass:ergo,ass:folibasic},
\[ \begin{split}
  \delta \int \Phi(x) d\mu(x) \, \overset{\mathrm{a.s.}}{=}\, 
   \lim_{W\rightarrow\infty} \lim_{L\rightarrow\infty} \frac 1L \sum_{n=1}^W  \sum_{l=1}^{L} \left( \Phi(X_{n+l}) - \Phi_{avg} \right)  
 \delta \tf_m(z_{l+1}) \cdot 
\frac{d_z p(z_{l+1}, x_{l+1})}{ p(z_{l+1}, x_{l+1})} \end{split} \]  
according to  $X_0, Y_1, Y_2,\cdots \sim \mu \times p\times p \times \cdots$.
Here $\Phi_{avg}:=\int\Phi d\mu$.
\end{theorem}

Further generalizations are possible, as we did in \Cref{s:generalize}.
In particular, we can let the conditional density $p$ also depend on $\gamma$.
This requires only adding the $\delta p$ term to the formula.

\section{Kernel-differentiation algorithm}
\label{s:algorithm}

This section gives the procedure list of the algorithm and demonstrates the algorithm in several examples.
We no longer label the subscript $\gamma$, and $\delta$ can be the derivative evaluated at $\gamma\neq 0 $; the dependence on $\gamma$ should be clear from context.
The codes used in this section are available at \url{https://github.com/niangxiu/np}.

\subsection{Procedure lists}
\label{s:procedure}

First, we give the procedure list for the kernel-differentiation algorithm for $\delta h_T$ of a finite-time system, whose formula was derived in \Cref{l:dhTn}.
Then we give the procedure list for the infinite-time system, whose formula was derived in \Cref{t:dh}.
The foliated cases have the same algorithm, and we only need to change to the suitable definition of $dp/p$.

First, we give the algorithm for the finite-time systems. 
Here $L$ is the number of sample paths whose initial conditions are generated independently from $h_0$ (or $\mu_0$, if the measure is singular and does not have a density).
This algorithm requires that we already have a random number generator for $h_0$ and $p$:
this is typically easier for $p$ since we tend to use simple $p$ such as Gaussian; but $h_0$ might require more advanced sampling tools.

\vspace{0.2in}\hrule\vspace{0.05in}
\begin{algorithmic}
\Require $L$, random number generators for densities $h_0$ and $p_m$.
\For{$l = 1, \dots, L$}
  \State Independently generate sample $x_0$ from $X_0\sim h_0$, 
  \For{$m = 0, \dots, T-1$}
    \State Independently generate sample $y_{m+1}$ from $Y_{m+1}\sim p_{m+1}$
    \State $I_{l,m+1} \gets \delta f x_{l,m} \cdot \frac {dp}{p} (y_{m+1}) $
    \State $x_{l,m+1} \gets f(x_{l,m}) + y_{m+1}$
  \EndFor
  \State $\Phi_{l,T}\gets\Phi_T (x_{l,T})$
  \Comment $\Phi_{l,T}$ takes less storage than $x_{l,T}$.
  \State $S_{l} \gets - \sum_{m=1}^{T}  I_{l,m}$
  \Comment No need to store $x_m$ or $y_m$.
\EndFor
\State{$\Phi_{avg,T} := \int \Phi_T h_{\gamma,T} dx 
\approx \frac 1L \sum_{l=1}^L \Phi_{l,T}$
}
\Comment{
Evaluated at $\gamma=0$.
}
\State $\delta \Phi_{avg,T} \approx \frac 1L \sum_{l=1}^L S_l \left( \Phi_T (x_{l,T}) -\Phi_{avg,T} \right)$ 
\end{algorithmic}
\vspace{0.05in}\hrule\vspace{0.2in}

Then we give the procedure list for the infinite-time case.
Here, $M_{pre}$ is the number of preparation steps during which $x_0$ lands on the physical measure.
Here, $W$ is the decorrelation step number, typically $W \ll L$, where $L$ is the orbit length.
Since here $h$ is given by the dynamic, we only need to sample the easier density $p$.

\vspace{0.2in}\hrule\vspace{0.05in}
\begin{algorithmic}
\Require $M_{pre}, W\ll L$, random number generator for $p$.
\State Take any reasonable $x_0$.
\For{$m = 0, \cdots, M_{pre}$} 
  \Comment{To land $x_0$ on the physical measure.}
  \State Independently generate sample $y$ from $Y \sim p$
  \State $x_0 \gets f(x_0)+y$
\EndFor
\For{$m = 0, \cdots, W+L-1$}
    \State Independently generate sample $y_{m+1} $ from $Y_{m+1} \sim p$
    \State $I_{m+1} \gets \delta f  x_{m} \cdot \frac {dp}{p} (y_{m+1}) $
    \State $x_{m+1} \gets f(x_m)+y_{m+1} $
    \State $\Phi_{m+1} \gets \Phi(x_{m+1})$
\EndFor
\State $\Phi_{avg}:=\int \Phi h_{\gamma} dx\approx \frac 1L \sum_{l=1}^L \Phi_l$
\State $\delta \Phi_{avg} \approx  - \frac 1L \sum_{n=1}^W  \sum_{l=1}^L \left(\Phi_{n+l} - \Phi_{avg} \right) I_{l+1}$ 
\end{algorithmic}
\vspace{0.05in}\hrule\vspace{0.2in}

From a utility point of view, the kernel-differentiation is an `adjoint' algorithm.
Since we do not compute the Jacobian matrix at all, so the most expensive operation per step is only to compute $f(x_m)$ and $\delta f(x_m)$, so the cost of computing the linear response of the first parameter (or $\delta f$) is less than the adjoint methods.
On the other hand, the marginal cost for a new parameter in this algorithm is low and is equal to that of adjoint methods.

We remind the readers of some useful formulas when we use Gaussian noise in $\R^M$.
Let the mean be $\mu\in\R^M$, and let the covariance matrix be $\Sigma$, which is defined by
$ \Sigma_{i,j} := \operatorname{E} [(Y_i - \mu_i)( Y_j - \mu_j)] = \operatorname{Cov}[Y_i, Y_j] $.
Then the density is
\[ \begin{split}
p (y) = (2\pi)^{-k/2}\det (\Sigma)^{-1/2} \, \exp { -\frac{1}{2} (y - \mu)^\mathsf{T} \Sigma^{-1}(y - \mu) }.
\end{split} \]
Hence we have
\[ \begin{split}
\frac{dp}p (y)
= - \Sigma^{-1} (y-\mu) 
\in \R^M.
\end{split} \]
Typically we use normal Gaussian, so $\mu = 0$ and  $\Sigma = \sigma^2 I_{M\times M}$, so we have
\begin{equation}\begin{split}\label{e:ttbm}
\frac{dp}p (y)
= - \frac 1{\sigma^2} y
\in \R^M.
\end{split}\end{equation}

If we consider the co-dimension-$c$ foliation by planes
$F_\alpha = \{x_1=\alpha_1, \ldots, x_{c} = \alpha_{c}
\}$,
the above formula is still valid.
For geometers, it might feel more comfortable to write $\delta\tf\cdot dp$ together, which is just the directional derivative of the function $p$ defined in the subspace.

\subsection{Numerical example: tent map with \re{additive noise}}
\label{s:tent}

We use our ergodic algorithm to compute the linear response of the physical measure of the tent map with additive noise.
The dynamics is
\[ \begin{split}
  X_{n+1}=f^\gamma(X_n)+Y_{n+1}
  \quad \textnormal{mod 1,} \quad 
  \quad \textnormal{where} \quad 
  \\
  Y_n\iid \cN(0,\sigma^2),
  \quad \textnormal{} \quad 
  f^\gamma(x) = \begin{cases}
    \gamma x 
    \quad\text{if}\quad 0\le x\le 0.5,
    \\
    \gamma(1-x)
    \quad\text{otherwise.}
  \end{cases}
\end{split} \]
In this subsection, we fix the observable
\[ \begin{split}
\Phi(x)=x
\end{split} \]

Previous linear response algorithms based on randomizing algorithms for deterministic systems (such as the fast response algorithm, path perturbation algorithm, and divergence algorithm) do not work in this example.

We shall demonstrate the kernel-differentiation algorithm in this example, and show that adding noise case can give a
\re{linear response which is useful for optimizing the deterministic case.
}
In the following discussions, unless otherwise noted, the default values for the (hyper-)parameters are
\[ \begin{split}
 \gamma=3, \quad
 \sigma=0.1, \quad 
 W=7. 
\end{split} \]

\begin{figure}[ht] \centering
  \includegraphics[height=7cm]{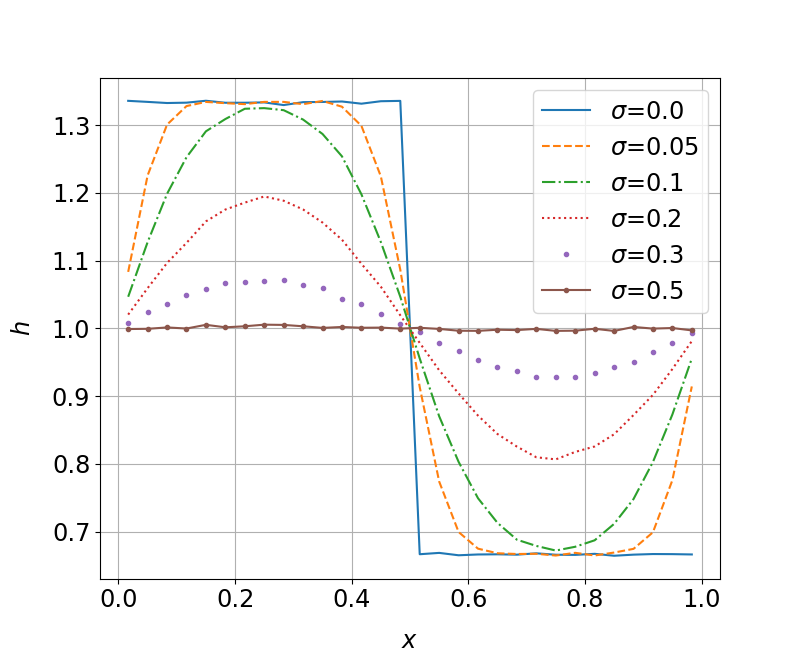} 
  \caption{
  Densities $h$ of different $\sigma$. 
  Here $L=10^7$.
  }
  \label{f:tentden}
\end{figure}

\begin{figure}[ht] \centering
  \includegraphics[scale=0.4]{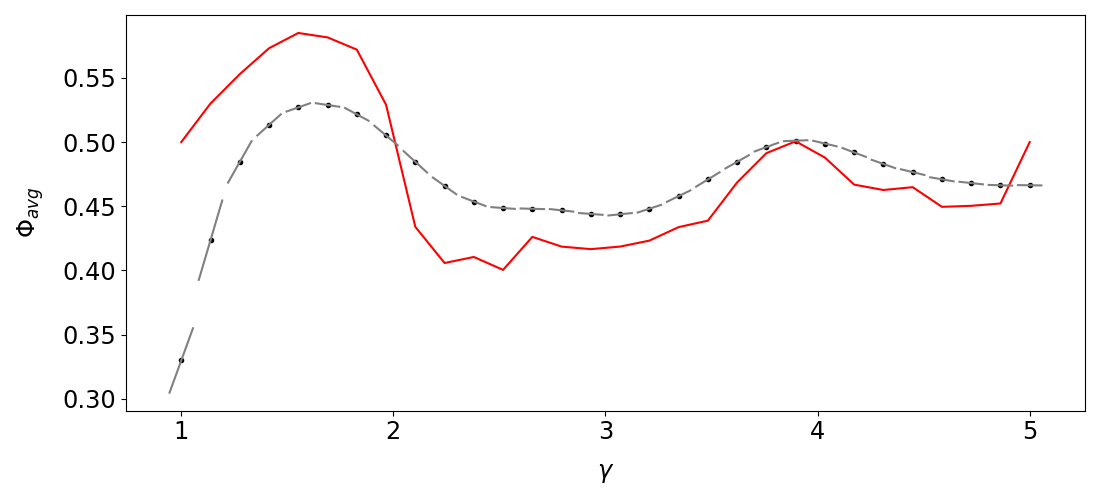}
  \caption{$\Phi_{avg}$ and $\delta \Phi_{avg} $ for different parameter $\gamma$.
  Here $L=10^6$ for all orbits.
  The dots are $\Phi_{avg}$, and the short lines are $\delta \Phi_{avg} $ computed by the kernel-differentiation response algorithm; they are computed from the same orbit.
  The long red line is $\Phi_{avg}$ when there is no noise ($\sigma=0$).
  }
  \label{f:opt}
\end{figure}

We first test the effect of adding noise on the physical measure.
As shown in \Cref{f:tentden}, the density converges to the deterministic case as $\sigma\rightarrow 0$.
Then we run the kernel-differentiation algorithm to compute linear responses for different $\gamma$.
As we can see in \Cref{f:opt}, the algorithm gives an accurate linear response for the noised case, which reveals the relation between the averaged observable $\Phi_{avg}=\int \Phi h dx$ and the parameter $\gamma$.
Moreover, the linear response in the noised case is a reasonable reflection of the observable-parameter relation of the deterministic case.
Most importantly, the local min/max of the noised and deterministic cases are close.
Hence, we can use the kernel-differentiation algorithm on the noised case to help optimizations in the deterministic case.

\begin{figure}[ht] \centering
  \includegraphics[width=0.45\textwidth]{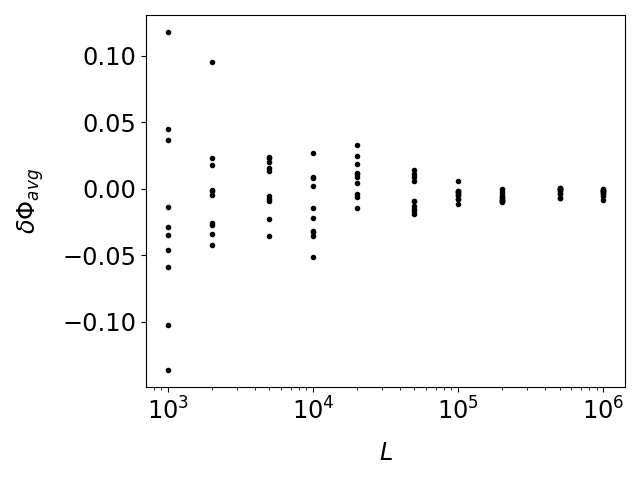}
  \includegraphics[width=0.45\textwidth]{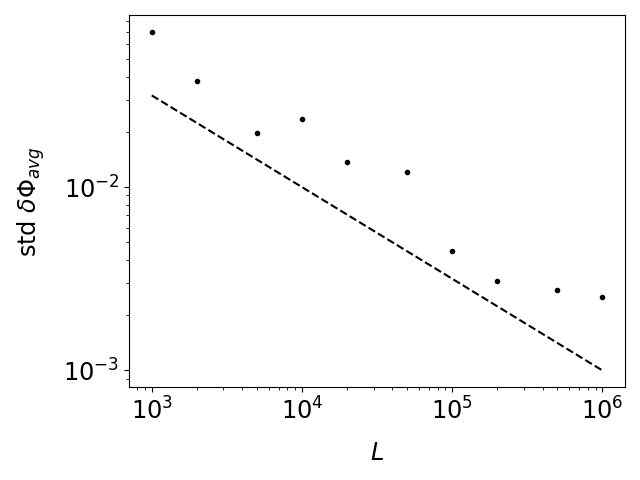}
  \caption{Effects of orbit length $L$.
  Left: derivatives from 10 independent computations for each $L$.
  Right: the standard deviation of the computed derivatives, where the dashed line is $L^{-0.5}$.}
  \label{f:tent_L}
\end{figure}

Then we show the convergence of the kernel-differentiation algorithm with respect to $L$ in \Cref{f:tent_L}.
In particular, the standard deviation of the computed derivative is proportional to $L^{-0.5}$.
This is the same as the classical Monte-Carlo method in probability, with $L$ being the number of samples.
This is expected: since here we are given the dynamics, our samples are canonically the points on a long orbit.

\begin{figure}[ht] \centering
  \includegraphics[width=0.45\textwidth]{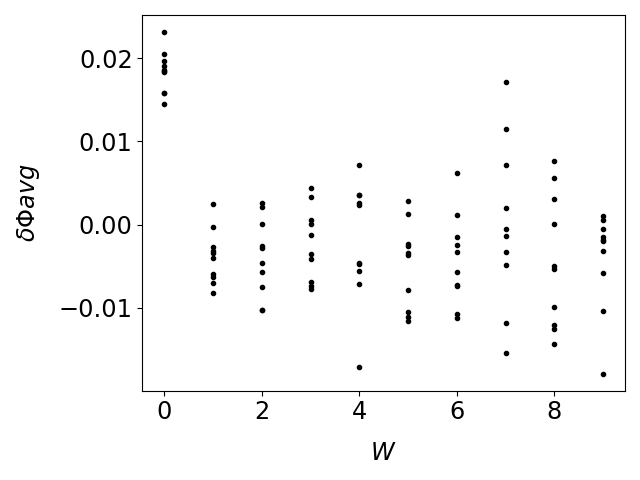}
  \includegraphics[width=0.45\textwidth]{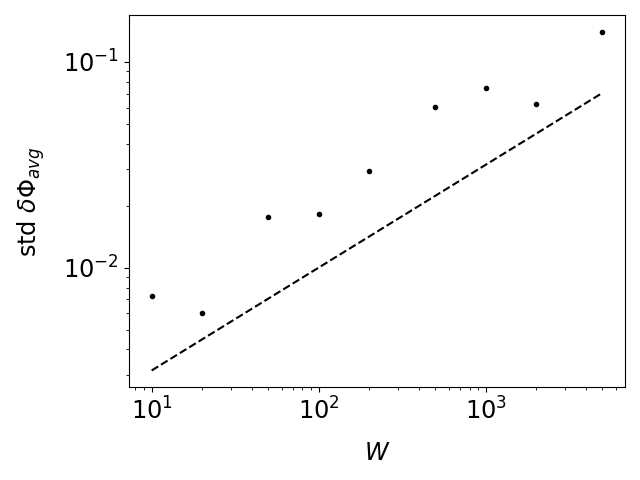}
  \caption{Tent map, effects of $W$. 
  Here $L=10^5$.
  Left: derivatives computed by different $W$'s.
  Right: standard deviation of derivatives, where the dashed line is $0.001W^{0.5}$.}
  \label{f:tent_W}
\end{figure}

\Cref{f:tent_W} shows that the bias in the average derivative decreases as $W$ increases, but the standard deviation increases roughly like $O(W^{0.5})$.
Note that if we do not centralize $\Phi$, then the standard deviation would be like $O(W)$.

\subsection{An unstable neural network with foliated perturbations}
\label{s:NN}

\subsubsection{Basic model}
\hfill\vspace{0.1in}

This subsection considers an unstable neural network whose dynamic is deterministic and whose exact linear response does not provide useful information.
Since the perturbation is foliated, we can add a low-dimensional noise, which incurs a smaller error.
Then we use the kernel-differentiation algorithm to compute the linear response of the modified system.

A main goal in the design of neural networks is to suppress gradient explosion, which is essentially the definition of unstable dynamics, or chaos.
However, even with modern architecture, there is no guarantee that gradient explosion (or chaos) can be precluded.
Unlike the conventional backpropagation method, the kernel-differentiation method does not perform propagation at all, so we are not hindered by gradient explosion.
However, we add extra noise at each layer, which introduces a (potentially large) systematic error.

\lil{yin}
\ree{
This example is also difficult for method which involves finite-element type method (this means to construct some approximation on a set of basis functions). 
Generally, the Monte-Carlo type method is more efficient when dimension $\ge4$ (see \cite[appendix A]{TrsfOprt} for a rough comparison).
This neural network example, with dimension $M=9$, is almost impossible for finite-element type method.
}

The variables of interest are $X'_n\in\R^9$, $n=0,1,\cdots, T$, where $T=50$, and the dynamic is 
\[ \begin{split}
  X'_{n+1}=f^\gamma(X'_n),
  \quad \textnormal{where} \quad 
  f^\gamma(x) = J \tanh (x + \gamma \one),
  \quad 
  J = CJ_0,
  \quad 
  \\
  J_0 = 
  \left[\begin{array}{ccccccccc}
 -0.54 &-1.19 &-0.33 & 1.66 &-0.5  &-1.3  & 1.52 &-0.5  & 1.95 \\
 -1.6  &-1.55 &-1.45 & 0.61 & 1.92 & 0.59 &-0.16 &-1.14 &-1.27 \\
 -0.59 &-0.65 &-1.32 &-1.46 &-0.82 &-0.95 &-1.47 &-0.08 &-0.38 \\
 -0.78 &-0.26 & 0.87 & 1.99 & 0.07 & 0.87 &-0.79 &-0.44 & 1.11 \\
  0.8  &-1.28 &-0.52 &-1.01 & 1.49 & 1.49 &-1.65 &-0.45 & 0.21 \\
 -1.77 & 0.03 &-1.39 &-0.28 & 0.44 & 1.27 & 0.61 & 0.01 &-0.02 \\
 -0.18 &-0.29 &-0.73 & 0.53 &-0.82 &-1.58 &-1.41 & 0.07 &-1.84 \\
  0.64 & 0.86 & 0.73 & 0.96 &-0.06 & 0.04 & 1.1  & 1.22 &-0.28 \\
  1.18 &-1.95 &-0.37 & 0.01 & 1.24 &-0.32 & 0.43 & 0.06 &-1.28
  \end{array}\right].
\end{split} \]
Here $I$ is the identity matrix, $\one=[1,\cdots, 1]$, and for $x =[x^1,\cdots,x^9], x^i\in\R$,
\[ \begin{split}
  \tanh(x)
  := [\tanh(x^1),\cdots,\tanh(x^9)].
\end{split} \]
We set $X'_0 \sim \cN(0,I)$.
The objective function is
\[
\Phi'(X'_T) = \sum_{i = 1}^9 (X'_T)^i.
\]

There is a somewhat tight region for $C$ such that the system is unstable: when $C<1$, then $J$ is small, so the Jacobian is small; when $C>10$, then the points tend to be far from zero, so the derivative of $\tanh$ is small, so the Jacobian is also small.
Our choice $C=4$ gives roughly the most unstable network.

We give a rough cost comparison with the ensemble (or stochastic gradient) formula of the linear response, which is
\[ \begin{split}
  \delta \E (\Phi')
  = \E \left(\sum_{n=1}^{50} \delta f(X'_{50-n}) \cdot f^{*n} d\Phi(X'_{50})\right)
\end{split} \]
Here $f^*$ is the backpropagation operator for covectors, which is the transpose of the is the Jacobian matrix $Df$.
On average $|f^{*50}| \approx 10^4\sim10^5$, $d\Phi\cdot \delta f\approx 10$, so the integrand size is about $10^5\sim 10^6$.
This would require approximately $10^{10} \sim 10^{12}$ samples (a sample is a realization of the 50-layer network) to reduce the sampling error to $O(1)$.
That is not affordable.

Our model is modified from its original form in \cite{Cessac04,Cessac06}.
In the original model, the entries in the weight matrix $J$ were randomly generated according to certain laws; as discussed in \Cref{s:rand}, we can rewrite the random maps by added noise, then obtain exact solutions to the original problem.
We can also further generalize this example to time-inhomogeneous cases.

Finally, we acknowledge that our neural network's architecture is outdated, but modern architectures are not good tests for our algorithm.
Because the backpropagation method does not work in chaos, current architectures typically avoid chaos.
The kernel-differentiation method might help us to have more freedom in choosing architectures beyond the current ones.

\subsubsection{Explicit foliated chart and artificial noise}
\hfill\vspace{0.1in}

\nax{
We show that the perturbation induced by changing the bias parameter is foliated.
}
We say that a perturbation is foliated if there is a fixed foliation such that for a small interval of $\gamma$, $\delta f$ is always parallel to that foliation.
The perturbation in bias is foliated, since for any $\gamma_1, \gamma_2$,
\[
\left. \pp f \gamma \right|_{\gamma = \gamma_1}(x_1)
= \left. \pp f \gamma \right|_{\gamma = \gamma_2}(x_2),
\quad\textnormal{whenever}\quad
f_{\gamma_1}(x_1)= f_{\gamma_2}(x_2).
\]
Hence, the vector field $\delta \tf$ is invariant for an interval of $\gamma$, and the foliation is given by the streamlines of $\delta f$.

In fact, we can write down the explicit expression of a foliated chart, that is, change the coordinate by $X = X' + \gamma\one$, then the dynamic and observable under the new coordinate is 
\[
  X_{n+1}=J \tanh (X_n) + \gamma\one,
  \quad
  \Phi(X_T) = -9\gamma + \sum_{i = 1}^9 X_T^i.
\]
Note that we can allow the chart to depend on $\gamma$.
In addition, the initial distribution is changed to
\[
X_0 = Y_0 + \gamma \one,
\quad\textnormal{where}\quad
Y_0 \sim \cN(0,I).
\]

Now it is clear that $\delta f$ is always in the direction of $\one$;
therefore, we only need to add noise in this direction.
We compute the linear response of the system
\begin{equation} \label{e:asd}
  X_{n+1}= J \tanh (X_n) + \gamma\one + Y_{n+1} \one/\sqrt{M},
  \quad \textnormal{where} \quad 
  Y_n \iid \cN(0,\sigma^2).
\end{equation}
We shall compare with adding noise in all directions, with $Y'\in \R^M$
\[
  X_{n+1} = J \tanh (X_n) + \gamma\one + Y'_{n+1},
  \quad \textnormal{where} \quad 
  Y'_n\iid \cN(0,\sigma^2 I).
\]
Here $I$ is the $M\times M$ identity matrix.

The foliated case has smaller noise added to the system, since $Y'$ equals summing i.i.d pieces of $Y$ in orthogonal directions for $M=9$ times.
Then we show that the two cases with noise have the same computational cost.
As discussed in \Cref{s:costErr}, the computational cost is determined by the number of samples required, which is further determined by the magnitude of the integrand.
To verify that the costs are the same, we just need to check that the size of the integrands are the same, that is, whether
\[
\E\left[\left(\frac 1{p'} \delta f \cdot dp'\right)^2\right]
= \E\left[\left(\frac 1p \delta f \cdot dp\right)^2\right],
\]
where $p$ is the density in the 1-dimensional subspace, and $p'$ is the density in all directions.
By \Cref{e:ttbm}, and note that $\delta f = \one$ is constant, we want
\[
\E\left[\left(\frac{1}{\sigma^2} \one \cdot Y'\right)^2\right]
= \E\left[\left(\frac{1}{\sigma^2} \one \cdot \one Y /\sqrt{M}\right)^2\right].
\]
This is indeed true; just notice that $\one /\sqrt{M}$ is the unit vector. 
Hence, for the same cost, the foliated case has a smaller noise error.

Note that now both $\Phi$ and $h_0$ depend on $\gamma$ due to the chart.
Hence, the linear response has two more terms,
\begin{equation} \label{e:rrr}
\delta \int \Phi(x_T) d\mu_T(x_T) 
= T +
\int \delta \Phi(x_T) d\mu_T(x_T) 
+ \E  \left[ \Phi(X_T) \frac{\delta h_0}{h_0}(X_0)\right],
\end{equation}
where the first term $T$ is given by \Cref{t:dhTnfoli}, and the other two terms are new.
For the second term, since the integrand $\delta \Phi = -9$ is constant, the second term is $-9$; if the integrand is not constant, then we can still easily perform Monte-Carlo sampling.
For the third term, the expectation $\E$ is just an abbreviation for multiple integration in \Cref{t:dhTnfoli}.
By definition,
\[
h_0(x_0)
= (2\pi)^{-k/2} \exp { -\frac{1}{2} (x_0-\gamma\one )^\mathsf{T} (x_0-\gamma\one) },
\]
so $\frac{\delta h_0}{h_0}(x_0) = x_0\cdot \one -\gamma M$.

\subsubsection{Numerical results}
\hfill\vspace{0.1in}

We use the kernel-differentiation method given in \Cref{t:dhTnfoli} to compute the main term in \Cref{e:rrr}, under two different noises: a noise in all directions, and a 1-dimensional noise along the direction of $\one$.

\begin{figure}[ht] \centering
  \includegraphics[scale=0.4]{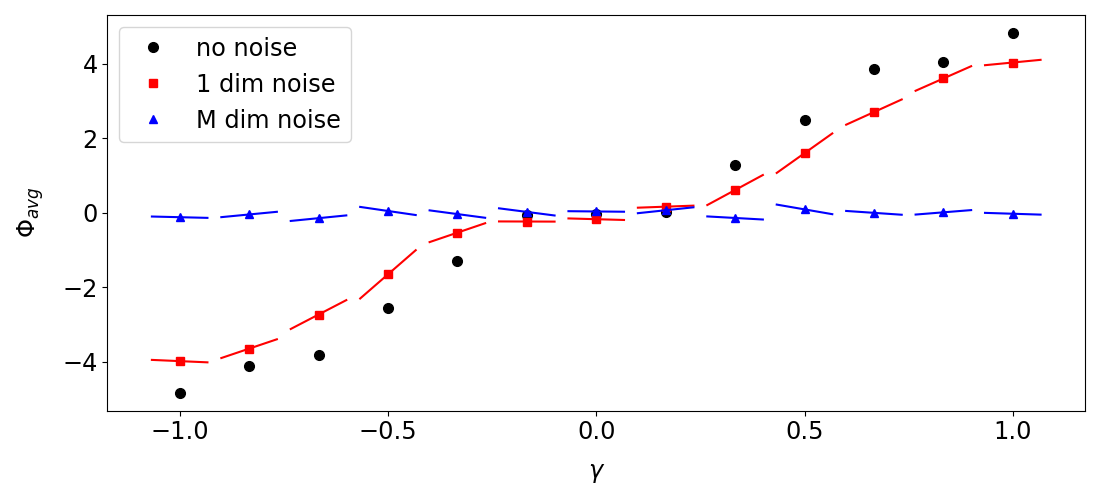}
  \caption{$\Phi_{avg}$ and $\delta \Phi_{avg} $ for different parameter $\gamma$ and different noises.
  For the two cases with noise, $\sigma = 1.5$.
  Here the total number of samples is $L=10^4$.
  }
  \label{f:NN1.5}
\end{figure}

\begin{figure}[ht] \centering
  \includegraphics[scale=0.4]{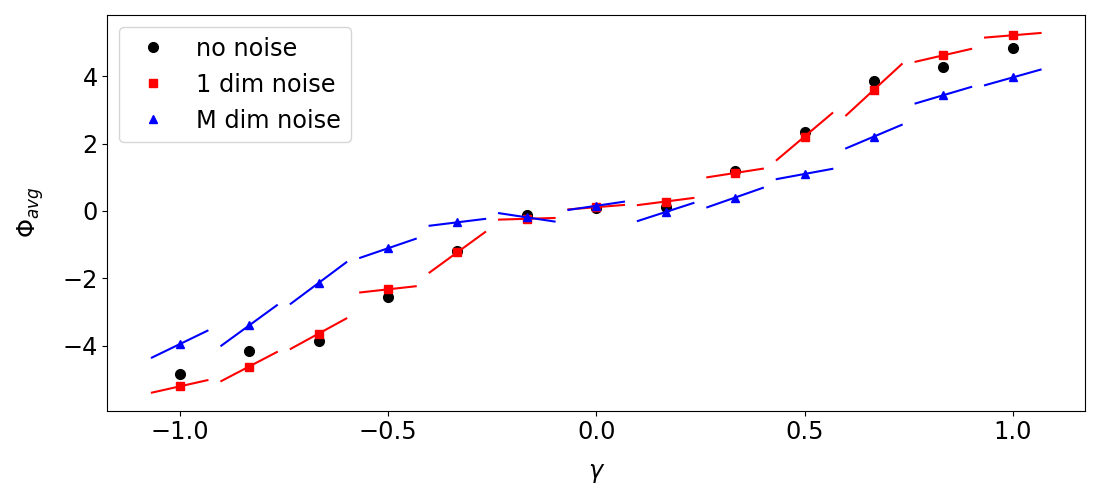}
  \caption{ $\sigma = 0.5$.
  Other settings are the same as \Cref{f:NN1.5}.
  }
  \label{f:NN0.5}
\end{figure}

\Cref{f:NN1.5} and \Cref{f:NN0.5} show the results of the kernel-differentiation algorithm on this example.
The algorithm correctly computes the derivative of the problem with added noise.
\ree{
The total time cost of running the algorithm on $L=10^4$ orbits to obtain a linear response, using a 1 GHz computer thread, is 3 seconds.
}

As we can see by comparing the two figures,
the case with 1-dimensional noise is a better approximation of the deterministic case, which is what we are actually interested in.
In particular, when $\sigma=1.5$, the $M$-dimensional noise basically hides away any trend between the parameter $\gamma$ and the average observable $\Phi_{avg}$.
In contrast, the 1-dimensional noise 
\re{
has a much smaller impact on the $\gamma\sim \Phi_{avg}$ relation.
If we want to optimize $\gamma$ for the smallest $\Phi_{avg}$, then adding one-dimensional noise is an acceptable modification to the original system, since the optimum has not changed much.
On the other hand, adding $M$-dimensional noise might change the system too much, and the modified system does not help to optimize the original system.
}

Also note that it is not free to decrease $\sigma$ for the purpose of reducing the systematic error.
As we can see by comparing the two figures and as we shall discuss in \Cref{s:costErr}, a small $\sigma$ would increase the sampling error, which further requires more samples, thus increasing the cost.
We must use certain structures, such as adding low-dimensional noise when the perturbation is foliated, to reduce the approximation error with respect to the deterministic case.

\section{Discussions}
\label{s:discuss}

The kernel-differentiation algorithm is robust, does not have systematic errors, has low cost per step, and is not cursed by dimensionality.
But there is a caveat: When the noise is small, we need many data for the Monte-Carlo method to converge.
Hence, we cannot expect to use the small $\sigma$ limit to obtain an easy approximation of the deterministic systems.

In this section, we first give a rough cost-error estimation of the problem, and estimate the number of samples $L$ and the desired noise intensity $\sigma$.
Then we discuss how to potentially reduce the cost by further combining with the fast response algorithm, which was developed for deterministic linear responses.

\subsection{A rough cost-error estimation}
\label{s:costErr}

When the problem is intrinsically random, the noise scale $\sigma$ has been fixed.
In this case, there are two sources of error.
The first is due to using a finite decorrelation step number $W$; this error is $O(\theta^W)$ for some $0<\theta<1$.
The second is the sampling error due to using a finite number $L$ of samples.
Assume that we use Gaussian noise; since we are averaging a large integrand to get a small number, we can approximate the standard deviation of the integrand by its absolute value.
The integrand is the sum of $W$ copies of $\frac{dp}p = -\frac y{\sigma^2}\sim \frac 1\sigma$, so the size is roughly $O(\sqrt W /\sigma)$.
So, the sampling error is $O(\sqrt W /\sigma\sqrt L)$, where $L$ is the number of samples.
Together we have the total error $\eps$
\[ \begin{split}
  \eps = O(\theta^W) + 
  O \left(\frac {\sqrt W} {\sigma\sqrt L} \right).
\end{split} \]

This gives us a relation among $\eps, W$, and $L$, where $L$ is proportional to the total cost.
In practice, we typically set the two errors be roughly equal,
which gives an extra relation for us to eliminate $W$ and obtain the cost-error relation
\[ \begin{split}
  \theta^W = \frac {\sqrt W} {\sigma\sqrt L},
  \quad \textnormal{so} \quad 
  L 
  = O \left( \frac { W} {\sigma^2 \theta^{2W}} \right)
  = O \left( \frac {\log _\theta \eps} {\sigma^2 \eps^{2}} \right)
  .
\end{split} \]
This is rather typical for Monte-Carlo method.
But the problem is that the cost can be large for small $\sigma$.

On the other hand, if we use a random system to approximate deterministic systems, then we can have the choice on the noise scale $\sigma$.
Now each step further incurs an approximation error $O(\sigma)$ on the physical measure.
This error decays, but accumulates, and the total error in the physical measure is $O(\sigma/(1-\theta))$.
Hence, if we are interested in the trend between $\Phi_{avg}$ and $\gamma$ for a certain step size $\Delta \gamma$ (this is typically known from the practical problem), then the error between the slope of the deterministic system and its random approximation is $O(\sigma/\Delta \gamma (1-\theta))$.
The slope of the random system can be computed by the kernel-differentiation method.
(This explains why the random system has the linear response whereas the deterministic system might not, since the error is large for $\Delta \gamma$ small.)
Together, the total error $\eps$ of the kernel-differentiation method, in terms of approximating the slope of deterministic systems over an interval of size $\Delta \gamma$, is
\begin{equation} \begin{split} \label{e:lian}
  \eps 
  = O(\theta^W) 
  + O \left(\frac {\sqrt W} {\sigma\sqrt L} \right)
  + O \left(\frac {\sigma} {\Delta \gamma (1-\theta)} \right).
\end{split} \end{equation}

Again, in practice we want the three errors to be roughly equal, which shall prescribe the size of $\sigma$,
\[ \begin{split}
  \sigma = \eps \Delta \gamma (1-\theta),
  \quad \textnormal{so} \quad 
  L 
  = O \left( \frac {\log _\theta \eps} {\sigma^2 \eps^{2}} \right)
  = O \left( \frac {\log _\theta \eps} {\eps^{4}  \Delta \gamma^2 (1-\theta)^2} \right).
\end{split} \]
Since $\Delta \gamma$ can be small, this cost can be much larger than just $\eps^{-4}$, which is already a high cost.

Finally, we acknowledge that our estimation is very inaccurate, but the point we make is solid, that is, the small noise limit is numerically expensive to obtain.

\subsection{\ree{Three basic ideas for linear response}}
\label{s:3lin}

This subsection hopes to convince readers that there are, to a degree, \textit{only} three most straightforward ideas in terms of computing the linear response of random dynamical systems:
the path-perturbation, the divergence, and the kernel-differentiation.
We derive the three formulas for one-step, which is enough to illustrate the differences and connections.
It also shows the possibility of combining all the basic ideas, which we discuss in \Cref{s:unify}.

The same as \Cref{s:prep}, let $X_0\sim h_0$, $Y\sim p$, the one-step dynamics is $X_1 = f^\gamma(X_0)+Y$.
Denote the expectation of $X_1$ as
\begin{equation} \begin{split} \label{e:start}
  \E[\Phi(X_1)]
  := \int \int \Phi(x_1) h_0(x_0)  p(y) dx_0 dy.
\end{split} \end{equation}
Note that here $x_1$ is a function of the two dummy variables $x_0$ and $y$, and depends on $\gamma$.
The three variables are related by
\[ \begin{split}
  x_1 = f^\gamma(x_0)+y.
\end{split} \]

If we differentiate \Cref{e:start} as is, then we obtain the one-step path-perturbation formula,
\[ \begin{split}
  \delta \E[\Phi(X_1)]
  = \int \int d \Phi(x_1) \delta f(x_0) h_0(x_0)  p(y) dx_0 dy.
\end{split} \]
Here, $d \Phi(x_1) \delta f(x_0)$ is the perturbation in $\Phi$ due to the perturbation in $x_1$, with dummy variables $x_0$ and $y_1$ fixed.
Note that here $\Phi$ is differentiated.
For many-step systems, we can still use this idea, that is, we first compute how each path is perturbed under quenched noise and initial condition, and then compute the perturbation in the path and then the perturbation in $\Phi$.
The problem is that when the system is unstable, this path-perturbation diverges.

However, we do not necessarily use $x_0$ and $y$ as dummy variables: There are (only) three ways to choose two out of $x_0, x_1$, and $y$!
If we change the variable $y$ to $x_1=f^\gamma(x_0) + y$, and use $x_0$ and $x_1$ as dummy variables, then
\[ \begin{split}
  \E[\Phi(X_1)]
  = \int \int \Phi(x_1) h_0(x_0)  p(x_1-f^\gamma(x_0)) dx_0 dx_1.
\end{split} \]
$y$ is the function of $x_0, x_1$ and $\gamma$.
Take derivative, then $p$ will be differentiated, and we get the kernel-differentiation formula for one-step.
For many-step systems, we can still use this idea, that is, we fix the orbit, and see how the probability density of getting this orbit is changed.
This is explained in detail in previous parts of this paper.

Finally, to obtain the divergence formula, we change the variable $x_0$ to $x_1=f^\gamma(x_0) + y$.
This causes a nontrivial Jacobian determinant,
\[ \begin{split}
  \left| \dd {x_1}{x_0} \right|
  = \left| \pp {f^\gamma(x_0)}{x_0} \right|
\end{split} \]
So
\[ \begin{split}
  \E[\Phi(X_1)]
  = \iint \Phi(x_1) h_0(x_0) \left| \pp {f^\gamma(x_0)}{x_0} \right|^{-1}  p(y) dy dx_1,
  \quad \textnormal{where} \quad
  x_0 = (f^\gamma)^{-1}(x_1-y).
\end{split} \]
For now, we assume $f^\gamma$ is injective; if it is $n$-to-1, then we need to further sum over its preimage.
Denote $\tf:=f^\gamma\circ f^{-1}$, by the definition of transfer operators, 
\[ \begin{split}
  \E[\Phi(X_1)]
  = \iint \Phi(x_1) (L_\tf L_{f}h_0)(x_1-y) p(y) dy dx_1.
\end{split} \]

When taking the derivative, we need to adopt the wisdom that `the derivative of the Jacobian determinant is typically a divergence'.
Intuitively, this is because the perturbation in the Jacobian is related to the change in density by the perturbation $\delta f$, which is further related to its divergence.
Mathematically, this wisdom is just a wrap over exterior algebra, and readers can see \cite{TrsfOprt} for more details and extension to (unstable) submanifolds.
Substituting this wisdom into the derivative, we arrive at
\[ \begin{split}
  \delta \E[\Phi(X_1)]
  = \iint \Phi(x_1) \delta L_{\tf} \left(L_{f_0} h_0 \right) (x_1-y) p(y) dy dx_1,
  \quad \textnormal{where} \quad
  \delta L_{\tf} q = \div (q \delta \tf).
\end{split} \]
Note that here $\delta \tf$ is differentiated.

These one-step linear response formulas are equivalent under integration by parts.
But they perform quite differently for numerical applications and in proofs, since they put derivative on different terms.
The difference is even sharper for multi-step and infinite-step systems: we reviewed these differences in \Cref{s:review,s:review2}.
However, over the years, the choice among the three basic ideas became almost subconscious.
Now we have recalled the difference and strength of these ideas, so we can better understand why and how to combine them.

\subsection{A potential program to unify three linear response formulas}
\label{s:unify}

We sketch a potential program on how to further reduce the cost/error of computing the approximate linear response (in terms of reflecting the parameter-observable relation) of nonhyperbolic deterministic systems.
As is known, nonhyperbolic systems do not typically have linear responses, so we must mollify, and in this paper we choose to mollify by adding noise in the phase space during each time step.
But, as we see in \Cref{s:costErr}, adding a large noise increases the noise error, the third term in \Cref{e:lian}; while adding a small noise increases the sampling error, the second term in \Cref{e:lian}.

There are cases with special structures, and we have specific tools which could give a good approximate linear response.
When the system has a foliated structure, this paper shows that we can add low-dimensional noise.
When the unstable dimension is low, computing only the shadowing or stable part of the linear response may be a good approximation (see \cite{Ruesha}).
When unstable modes are of high frequency, Fang and Papadakis showed that we can use low-pass filters to remove unstable modes from shadowing solutions, further improving computational efficiency \cite{FP25}.

For the more general situation, a plausible solution is to add a large but local noise, only at locations where the hyperbolicity is bad.
The benefit of this program is that, if the singularity set is low-dimensional, then the area where we add big noise is small, and the noise error is small.
For the continuous-time case, there seems to be an easy choice: we can let the noise scale be reverse proportional to the flow vector length.
This should at least solve the singular hyperbolic flow cases \cite{Viana2000,Wen2020}, where the bad points coincide with zero velocity.
But for the discrete-time case, it can be difficult to find a natural criteria which is easy to compute.

In terms of formulas, we use the kernel-differentiation formula (the generalized version in \Cref{s:generalize}) where the noise is large, so the sampling error is also small.
Where the noise is small, we use the fast-response algorithm, which is efficient regardless of noise scale, but requires hyperbolicity.
This program hinges on the assumption that the singularity set is small or low-dimensional.

This program also requires us to invent more techniques.
For example, we need a formula that transfers information from the kernel-differentiation formula to the fast response formulas.
For example, the equivariant divergence formula, in its original form, requires information from the infinite past and future \cite{TrsfOprt}.
But we can rerun the proof and restrict the time dependence to finite time; this requires extra information on the interface, such as the derivative of the conditional measure and the divergence of the holonomy map.
These information should and could be provided by the kernel-differentiation formula. 
On the other hand, I recently gave the path-kernel formula, which combines kernel-differentiation with path-perturbation \cite{dud}.

Historically, there are three famous linear response formulas.
The path-perturbation formula does not work for chaos; 
the divergence formula is likely to be cursed by dimensionality;
the kernel-differentiation formula is expensive for small noise limit.
The fast response formula combines the path-perturbation and the divergence formula, it is a function (not a distribution) so can be sampled;
it is neither affected by chaos or high-dimension, but is rigid on hyperbolicity.
In the future, besides trying to test above methods on specific tasks, we should also try to combine all three formulas together into one, which might provide the best approximate linear responses with highest efficiency.

\section*{Acknowledgements}

The author is in great debt to Stefano Galatolo, Yang Liu, Caroline Wormell, Wael Bahsoun, and Gary Froyland for helpful discussions.

\section*{Data availability statement}
The code used in this manuscript is at \url{https://github.com/niangxiu/kernelDiff}.
There is no other associated data.



\bibliographystyle{abbrv}
{\footnotesize\bibliography{library}}

\end{document}